\documentclass{article}
\usepackage[%
journal=JNCG,    %% Replace XXX by one of the above journal codes.
lang=american,   %% Change to `american' if you use American English.
]{ems-journal}
\usepackage{amscd}
\usepackage[normalem]{ulem}
\usepackage{mathtools}
\usepackage{mathrsfs}

\newtheorem{thm}{Theorem}[section]
\newtheorem{lemm}[thm]{Lemma}
\newtheorem{prop}[thm]{Proposition}
\newtheorem{cor}[thm]{Corollary}

\theoremstyle{definition}

\numberwithin{equation}{section}

\newcommand{\bC}{{\mathbb C}}

\newcommand{\bF}{{\mathbb F}}

\newcommand{\bM}{{\mathbb M}}
\newcommand{\bN}{{\mathbb N}}

\newcommand{\bQ}{{\mathbb Q}}
\newcommand{\bR}{{\mathbb R}}

\newcommand{\bZ}{{\mathbb Z}}

\newcommand{\cH}{{\mathcal H}}

\newcommand{\cO}{{\mathcal O}}

\newcommand{\cU}{{\mathcal U}}

\DeclareMathOperator{\tr}{tr}

\DeclareMathOperator{\Tr}{Tr}

\DeclareMathOperator{\Span}{span}

\DeclareMathOperator{\Inn}{Inn}

\DeclareMathOperator{\Det}{det}

\DeclareMathOperator{\rank}{rank}

\DeclareMathOperator{\Der}{Der}

\DeclareMathOperator{\ev}{ev}

\DeclareMathOperator{\op}{op}
\DeclareMathOperator{\orb}{orb}
\DeclareMathOperator{\sa}{sa}

\DeclarePairedDelimiter{\norm}{\lVert}{\rVert}
\DeclarePairedDelimiter{\ang}{\langle}{\rangle}

\begin{document}

\title{Vanishing first cohomology and strong 1-boundedness for von Neumann algebras}

\emsauthor{1}{Ben Hayes}{B.~Hayes}
\emsaffil{1}{Department of Mathematics, University of Virginia, 
141 Cabell Drive, Kerchof Hall,
P.O. Box 400137
Charlottesville, VA 22904\email{brh5c@virginia.edu}}

% \email{
% \urladdr{https://sites.google.com/site/benhayeshomepage/home}
\emsauthor{2}{David Jekel}{D.~Jekel}
\emsaffil{2}{Department of Mathematics, University of California,
San Diego, 9500 Gilman Drive \# 0112, La Jolla, CA 92093
\email{djekel@ucsd.edu}}
% \urladdr{http://davidjekel.com}

\emsauthor{3}{Srivatsav Kunnawalkam Elayavalli}{S.~Kunnawalkam Elayavalli}
\emsaffil{3}{Institute of Pure and Applied Mathematics, UCLA, 460 Portola Plaza, Los Angeles, CA 90095, USA
\email{srivatsav.kunnawalkam.elayavalli@vanderbilt.edu}}
% \urladdr{https://sites.google.com/view/srivatsavke}Institute of Pure and Applied Mathematics, UCLA, 460 Portola Plaza, Los Angeles, CA 90095, USA

% \thanks{B. Hayes gratefully acknowledges support from the NSF grant DMS-2000105.  D. Jekel was supported by NSF grant DMS-2002826.}

\classification[22D55]{46L54}


\begin{abstract}
 In this paper, we obtain a new proof of a result of Shlyakhtenko which states that if $G$ is a sofic, finitely presented group with vanishing first $\ell^2$-Betti number, then $L(G)$ is strongly 1-bounded. Our proof of this result adapts and simplifies Jung's technical arguments which showed strong 1-boundedness under certain conditions on the Fuglede-Kadison determinant of the matrix capturing the relations. Our proof also features a key idea due to  Jung which involves an iterative estimate for the covering numbers of microstate spaces. We also provide a short proof using works of Shlyakhtenko and Shalom that the von Neumann algebras of sofic groups with Property T are strongly 1 bounded, which is a special case of another result by the authors.
%The notion of strong 1-boundedness for  finite von Neumann algebras was introduced in \cite{Jung2007}. This framework provides a novel free probabilistic approach to study rigidity properties and classification of finite von Neumann algebras.  We show that if $G$ is finitely presented, sofic and satisfies $\beta_1^{(2)}(G)=0$, then $L(G)$ is strongly 1-bounded, thus recovering the result of Shlyakhtenko \cite{Shl2015} with a purely microstates approach. The key technique is to find iterative packing estimates  of microstate spaces, in the spirit of Jung. Using this technique, we are able to show that all tracial von Neumann algebras with a finite Kazhdan generating set are strongly 1-bounded. This includes all Property (T) factors, and all group von Neumann algebras of Property (T) groups. This result generalizes   previous results in this direction due to Voiculescu, Ge-Shen, Jung-Shlyakhtenko, and Connes-Shlyakhtenko. We also settle the question of strong 1-boundedness for many inner amenable groups $G$, namely, whenever $G$ is finitely presented, non-sofic, or torsion-free.
\end{abstract}

\maketitle

% Another version of abstract:

\section{Introduction}

A \emph{tracial von Neumann algebra} is a pair $(M,\tau)$ of a finite von Neumann algebra and a faithful normal tracial state.  For every group $G$, there is an associated tracial von Neumann algebra, the von Neumann algebra $L(G)$ generated by the left regular representation of $G$ on $\ell^2(G)$ with the trace given by $\ang{\delta_e, (\cdot) \delta_e}$, and a major theme of operator algebraic research has been how the properties of a group (algebraic, analytic, geometric, etc.) are reflected by its von Neumann algebra.

In particular, one may consider finitary approximations of the group in several senses:  A group is \emph{sofic} if the group trace can be approximated by almost-representations in permutation groups; on the other hand, $L(G)$ is \emph{Connes-embeddable} if the same holds for representations in unitary groups, or if the group can be approximated by $*$-representations in matrices.  Voiculescu's free entropy dimension was introduced to quantify the \emph{amount} of approximations by matrices there are for any tuple $x$ in a von Neumann algebra \cite{VoiculescuFreeEntropy2, Voiculescu1996}.  The standard generators for a free group $\mathbb{F}_n$ for $n \geq 2$ have many approximations, and Voiculescu used this fact to deduce that the von Neumann algebra has no Cartan subalgebras \cite{Voiculescu1996}.  The free entropy approach has had several other applications to free group von Neumann algebras (and more generally free products) \cite{ GePrime, DykemaFreeEntropy, Jung2007, Hayes2018,Shl2015}.   A related notion of \emph{strong $1$-boundedness} was introduced by Jung (see \cite{JungS}); this is a strengthening of the condition of having free entropy dimension $1$, with the useful property that it is independent of the choice of generating set.  The first author reformulated strong $1$-boundedness through a numerical invariant $h$, called the $1$-bounded entropy, which is finite if and only if $M$ is strongly $1$-bounded (see \cite{Hayes2018}).

The work of Jung \cite{JungL2B} and Shlyakhtenko \cite{Shl2015} investigated strong $1$-boundedness in the context of polynomial relations in a $*$-algebra.
% with $\ell^2$-cohomology for groups.
In particular, Shlyakhtenko explicitly connected this to $\ell^{2}$-Betti numbers of groups.
In this paper we present an alternative proof of Shlyakhtenko's result \cite{Shl2015} that finitely presented sofic groups with vanishing first $\ell^2$-Betti number are strongly $1$-bounded (results in this direction are in Jung's paper but under somewhat restrictive hypotheses).  Shlyakhtenko's result generalized Jung's earlier work \cite{JungL2B} but with a different proof strategy using non-microstates free entropy rather than microstates free entropy.  We give a purely microstates proof that streamlines Jung's original ideas and clarifies the essential ingredients and limitations of this approach. %see \S \ref{subsec: L2 Betti discussion} for further discussion.
The statement of the theorem is as follows.

\begin{thm}[{\cite{Shl2015}}] \label{thm: vanishing L2 Betti}
If $G$ is a sofic finitely presented group with vanishing first $\ell^2$-Betti number, then $L(G)$ is strongly $1$-bounded.
\end{thm}

%Jung \cite{JungL2B} and Shlyakhtenko \cite{Shl2015} proved Theorem \ref{thm: vanishing L2 Betti} using the following more free probabilistic result.

Shlyakhtenko proved Theorem \ref{thm: vanishing L2 Betti} by obtaining a key technical free probabilistic fact involving non microstates theory and Fisher's information. Using this in combination with an inequality between the microstates and non microstates free entropy dimensions, he obtains as a corollary, the following generalization of a result of Jung \cite[Theorem 6.9]{JungL2B}.
% the following fact originally due to Jung:

\begin{thm} \label{thm: von Neumann algebraic}
Let $(M,\tau)$ be a tracial $\mathrm{W}^*$-algebra generated by some $x \in M_{\sa}^d$.  Suppose $\norm{x}_\infty < R$. Let $m \in \bN \cup \{\infty\}$.  Let $f(t_1,\dots,t_d) \in \bC\ang{t_1,\dots,t_d}^{\oplus m}$ be a tuple of non-commutative polynomials such that $f(x) = 0$.  Let
\[
D_f(x) = \begin{pmatrix} x_1 \otimes 1 - 1 \otimes x_1 & \dots & x_d \otimes 1 - 1 \otimes x_d \\ \partial_{x_1} f(x) & \dots & \partial_{x_d} f(x) \end{pmatrix} \in M_{m,d}(M \otimes M^{\op}),
\]
and let $\mu_{|D_f|}$ be the spectral measure of $|D_f| = (D_f^* D_f)^{1/2}$ with respect to $\tau \otimes \tau$.  If
\begin{equation} \label{eq: integrable}
\int_{[0,\infty)} |\log t|\,d\mu_{|D_f|}(t) < \infty,
\end{equation}
(with the convention that $\log(0)=-\infty$)
then $M$ is strongly $1$-bounded.
\end{thm}

Strictly speaking, both Jung and Shlyakhtenko's results are about $\alpha$-boundedness for general $\alpha\geq 1$, whereas the above Theorem just covers $\alpha=1$. However, the  case of $\alpha=1$ is of the most interest in applications, and the case of $\alpha>1$ will not be relevant in our paper.

One deduces Theorem \ref{thm: vanishing L2 Betti} from Theorem \ref{thm: von Neumann algebraic} through the well-known relationship between group cocycles and derivations on the group algebra.  One then parameterizes the derivations in terms of their action on a self-adjoint generating set, hence obtains a bijection between derivations and vectors $z$ in the kernel of $\partial f(x)$.  Looking at cocycles that are orthogonal to the inner cocycles results in the additional condition of $\sum_{j=1}^d [x_j, z_j] = 0$, or that $z$ is in the kernel of the commutator operator in the first row of the matrix $D_f(x)$.  Hence, the first $\ell^2$-Betti number with respect to $\tau$ of the $*$-algebra generated by $x$ is the Murray-von Neumann dimension of the kernel of $D_f$.

The condition \eqref{eq: integrable} is needed for the microstates argument to go through in the case of \cite{JungL2B}, or the more general nonmicrostates estimate used in \cite{Shl2015} (the arguments are substantially different).  This hypothesis is nontrivial to check in the group case, and this is where one uses the assumption of soficity.  The bound \eqref{eq: integrable} expresses positivity of a certain Fuglede-Kadison determinant, which is known for sofic groups \cite{ElekSzaboDeterminant}.  We remark that Shlyakhtenko's results about vanishing $L^2$-Betti numbers have been generalized to $*$-algebras that are not group algebras \cite{BrannanVergnioux}, but this still requires some way of controlling the Fuglede-Kadison determinant.%  Thus, while Theorem \ref{thm: vanishing L2 Betti} can be used to show strong $1$-boundedness for sofic Property (T) groups, there is little hope of adapting this method to general Property (T) von Neumann algebras (or even general Property (T) groups) without some analog of soficity.

Our proof of Theorem \ref{thm: von Neumann algebraic} is longer than Shlyakhtenko's argument but it is more self-contained.  Indeed, Shlyakhtenko's argument used the external fact that $\chi \leq \chi^*$ from \cite{BCG2003} and the result about strong $1$-boundedness and non-amenability sets from \cite[Proposition A.16]{Hayes2018}.  In this paper, we generalize and streamline Jung's strategy from \cite{JungL2B}, which uses iteration to bound covering numbers for smaller and smaller $\varepsilon$ with errors controlled by the integral \eqref{eq: integrable}.  Much of the technical challenge in Jung's work had to do with converting between covering numbers with respect to different non-commutative $L^p$-norms on the von Neumann algebra (and in fact $L^p$ quasinorms for $p \in (0,1)$).  Our argument works mostly with $L^2$ norms but requires conversion between $L^1$ and $L^2$ norms at one point, and this is main time we use a significant external ingredient, Szarek's estimates for the covering numbers of Grassmannians \cite{Szarek}.

Another notable feature of the proof is the way in which the condition $\sum_j [x_j,z_j] = 0$ (which corresponded in cohomology to looking at cocycles orthogonal to inner cocycles) arises naturally in the microstate setting by considering the elements in a unitary orbit closest to a given point $x$.

We also remark that polynomials in Theorem \ref{thm: von Neumann algebraic} can be replaced more generally by power series, and even non-commutative trace $C^2$ functions in the sense of \cite{JLS2021}; see Remark \ref{rem:tracesmooth}.

%\subsection{Organization of the paper}

%We close by discussing the organization of the paper. We start in Section \ref{sec:preliminary} by giving some preliminary definitions on tracial von Neumann algebras, noncommutative laws, and the $1$-bounded entropy. This includes a discussion of permanence properties of the $1$-bounded entropy and a list of examples of algebras with nonpositive $1$-bounded entropy. In Section  \ref{sec: Property (T)} we prove Theorem \ref{thm: Property T} that Property (T) algebras with a finite Kazhdan set are strongly $1$-bounded. Section \ref{sec:d sums and matrix amp} studies the behavior of $h$ under direct sums and matrix amplification, proves Proposition \ref{prop: dichotomy}, and shows how to recover Jung and Shlyakhtenko's result that Property (T) von Neumann algebras have free entropy dimension at most $1$ from Theorem \ref{thm: Property T}.  Section \ref{sec: vanishing L2B} then gives our reproof of Theorem \ref{thm: von Neumann algebraic} due to Jung and Shlyakhtenko.  In Appendix \ref{sec: derivations}, we recall why Theorem \ref{thm: von Neumann algebraic} implies Theorem \ref{thm: vanishing L2 Betti} after giving relevant background on cocycles and derivations and soficity.  We then explain in Section \ref{sec:PropTsofic} how one can use results of Shlyakhtenko  \cite{Shl2015} and Shalom \cite[Theorem 6.7]{ShalomCohomolgyCharacterize} to show that every sofic Property (T) group is strongly $1$-bounded.

% \subsection*{Acknowledgements}

\section{Background}\label{sec:preliminary}

\subsection{Tracial von Neumann algebras and non-commutative laws}

A tracial von Neumann algebra is a pair $(M,\tau)$ where $M$ is a von Neumann algebra and $\tau\colon M\to \bC$ is a faithful, normal, tracial state.  The classical example is $\mathbb{M}_n(\bC)$ as a tracial von Neumann algebra with the tracial state $\tr_n$ given by
\[
\tr_{n}(A) = \frac{1}{n}\sum_{i=1}^{n}A_{ii}.
\]
We will primarily be interested in cases where $M$ is \emph{diffuse}, i.e. it has no nonzero minimal projections. The above algebra is finite-dimensional, and is thus not diffuse.
One interesting class of \emph{diffuse} tracial von Neumann algebras are the \emph{group von Neumann algebras}. For a discrete group, we define the \emph{left regular representation} $\lambda\colon G\to \mathcal{U}(\ell^{2}(G))$ by
\[(\lambda(g)\xi)(h)=\xi(g^{-1}h)\mbox{ for all $g,h\in G$.}\]
The \emph{group von Neumann algebra of $G$} is then
\[L(G)=\overline{\Span\{\lambda(g):g\in G\}}^{SOT}.\]
%We leave it as an exercise to show that
The linear functional $\tau\colon L(G)\to \bC$ given by $\tau(x)=\langle{x\delta_{1},\delta_{1}\rangle}$ is a faithful, normal, tracial state (see e.g. \cite[Remark 6.7.3]{KadisonRingroseII}). So $(L(G),\tau)$ is a tracial von Neumann algebra. Moreover, it  can be shown that $L(G)$ is diffuse if and only if $G$ is infinite.

% An interesting class of tracial von Neumann algebras are the \emph{group von Neumann algebras}.
% Given a  discrete group $G$ the \emph{left regular representation} $\lambda\colon G\to \mathcal{U}(\ell^{2}(G))$ is given by
% \[(\lambda(g)\xi)(h)=\xi(g^{-1}h)\mbox{ for $\xi\in\ell^{2}(G)$,$g,h\in G$.}\]
% The group von Neumann algebra $L(G)$ is defined to be \[
% \overline{\Span\{\lambda(g):g\in G\}}^{SOT}.
% \]
% The group von Neumann algebra can be turned into a tracial von Neumann algebra by defining $\tau\colon L(G)\to \bC$ by $\tau(x)=\ang{x(\delta_{1}),\delta_{1}}$. We may also view $\mathbb{M}_n(\bC)$ as a tracial von Neumann algebra with the tracial state $\tr_n$ given by
% \[
% \tr_{n}(A) = \frac{1}{n}\sum_{i=1}^{n}A_{ii}.
% \]
% The group von Neumann algebra is diffuse if and only if $G$ is infinite.
For a von Neumann algebra $M$, we use $M_{\sa}$ for the self-adjoint elements of $M,$ and $\mathcal{U}(M)$ for the unitary elements of $M$.

Abelian tracial von Neumann algebra correspond exactly to probability spaces, and so we may think of tracial von Neumann algebras as an instance of \emph{noncommutative probability spaces}. Mimicking the abelian case, for a tracial von Neumann algebra $(M,\tau)$ and $1\leq p\leq\infty$, we define $\|\cdot\|_{p}$ on $M$ by
\[\|x\|_{p}=\tau(|x|^{p})^{1/p},\mbox{ where $|x|=(x^{*}x)^{1/2}$.}\]
It can be shown \cite{Dixmier1953} that this is indeed a norm on $M$.  We use the notation $\norm{x}_\infty$ for the operator norm. More generally, for $x\in M^{d}$ we set
\[
\norm{(x_1,\dots,x_d)}_p = \begin{cases} \left( \sum_{j=1}^d \tau(|x_j|^p) \right)^{1/p}, & p \in [1,\infty), \\ \max_{j=1,\dots,d} \norm{x_j}, & p = \infty. \end{cases}
\]

% Since  tracial von Neumann algebras $(M,\tau)$ with $M$ abelian correspond precisely to probability spaces, we may think of tracial von Neumann algebras as (a special case of) \emph{noncommutative probability spaces}.
% Following this intuition, given a tracial von Neumann algebra $(M,\tau)$ and  $1\leq p <\infty$, we define the $\|\cdot\|_{p}$ on $M$ by
% \[
% \|x\|_{p}=\tau(|x|^{p})^{1/p},\mbox{ where $|x|=(x^{*}x)^{1/2}$.}
% \]
% It can be shown \cite{Dixmier1953} that this is indeed a norm on $M$.  We use the notation $\norm{x}_\infty$ for the operator norm.  Moreover, the definition of the norms can be extended to tuples by
% \[
% \norm{(x_1,\dots,x_d)}_p = \begin{cases} \left( \sum_{j=1}^d \tau(|x_j|^p) \right)^{1/p}, & p \in [1,\infty), \\ \max_{j=1,\dots,d} \norm{x_j}, & p = \infty. \end{cases}
% \]

If $(M,\tau)$ is viewed as a non-commutative probability space, then its elements maybe viewed as non-commutative random variables.  In fact, a $d$-tuple $x = (x_1,\dots,x_d) \in M_{\sa}^d$ is the non-commutative analog of an $\bR^d$-valued random variable.  In the commutative setting, a $\bR^d$-valued random variable naturally gives rise to a \emph{probability distribution} as a classical measure on $\bR^{d}$. It is not possible to define such a measure in the non-commutative setting. However, as probability measures of compactly supported measures may be uniquely characterized by their moments, we define an analog this probability distribution as a linear functional on non-commutative polynomials.
% Although one cannot define the law (or probability distribution) of $x$ as a classical measure, we may define its non-commutative law as a certain linear function on a non-commutative polynomial algebra, just like the probability distribution $d$-tuple $X$ of bounded classical random variables defines a map $\bC[t_1,\dots,t_d] \to \bC$ sending $p$ to $\mathbb{E}[p(X)]$

For $d \in \bN$, we let $\bC\ang{t_1,\cdots,t_d}$ be the algebra of noncommutative polynomials in $d$ formal variables $t_1$, \dots, $t_d$, i.e. the free $\bC$-algebra with $d$-generators. We give $\bC\ang{t_{1},\cdots,t_{d}}$ the unique $*$-algebra structure which makes the $t_{j}$ self-adjoint. By universality, if $A$ is any $*$-algebra, and $x=(x_{1},\cdots,x_d)\in A^d$ is a self-adjoint tuple, then there is a unique $*$-homomorphism $\bC\ang{t_{1},\cdots,t_d}\to A$ which sends $t_{j}$ to $x_{j}$. For $p \in \bC\ang{t_{1},\cdots,t_d}$ we use $p(x)$ for the image of $p$ under this $*$-homomorphism.  Given a tracial von Neumann algebra $(M,\tau)$ and $x \in M^d_{\sa}$, we define the \emph{law of $x$}, denoted $\ell_{x}$, to be the linear functional $\ell_{x}\colon \bC\ang{t_1,\cdots,t_d}\to \bC$ given by
\[
\ell_{x}(f)=\tau(f(x)).
\]

The non-commutative laws can be characterized as follows.

\begin{prop}[{See \cite[Proposition 5.2.14]{AGZ2009}}] \label{prop:NClaws}
The following are equivalent.

\begin{enumerate}[(i)]
    \item There exists a tracial von Neumann algebra $(M,\tau)$ and $x \in M_{\sa}^d$ such that $\ell = \ell_x$ and $\norm{x}_\infty \leq R$. \label{item:concrete law characterization}
    \item $\ell$ satisfies the following conditions:  \label{item:abstract characterization of law}
    \begin{itemize}
        \item $\ell(1) = 1$,
        \item $\ell(f^*f) \geq 0$ for $f \in \bC\ang{t_1,\dots,t_d}$,
        \item $\ell(fg) = \ell(gf)$ for $f, g \in \bC\ang{t_1,\dots,t_d}$,
        \item $|\ell(t_{i_1} \dots t_{i_k})| \leq R^k$ for all $k \in \bN$ and $i_1$, \dots, $i_k \in \{1,\dots,d\}$.
        \end{itemize}
    \end{enumerate}
    \end{prop}

    For $R>0$,$d\in \bN$, we let $\Sigma_{d,R}$ be the space of laws satisfying either of the above equivalent conditions (for this specific) $R$.
We also denote $\Sigma_d = \bigcup_{R > 0} \Sigma_{d,R}$. Since $\Sigma_{d}$ is a space of linear functionals on $\bC\ang{t_{1},\cdots,t_{d}}$ we can give ti the weak$^{*}$-topology.

\textbf{Remark.} \label{rem:lawGNS}
The proof of (\ref{item:abstract characterization of law}) implies (\ref{item:concrete law characterization}) uses the GNS construction (see \cite[Proposition 5.2.14]{AGZ2009}).
Namely, let $H = L^2(\ell)$ be separation-completion of $\bC\ang{t_1,\dots,t_d}$ with respect to the semi-inner product $\langle{f,g\rangle}_\ell = \ell(f^*g)$. It can be shown that multiplication by $t_{j}$ is bounded with respect to this semi-inner product, and induces a well-defined bounded, self-adjoint operator $x_{j}$ on $L^{2}(\ell)$. Let $M=W^{*}(x_{1},\cdots,x_{d})$ be the algebra \emph{generated} by $x_{1},\cdots,x_{d}$, and we define $\tau\colon M\to \bC$ by $\tau(x)=\langle{x 1,1\rangle}$, where $1\in \bC\langle{t_{1},\cdots,t_{d}\rangle}$ is viewed as a vector in $L^{2}(\ell)$.
We will denote $M=W^{*}(\ell)$, and $\pi_{l}\colon \bC\langle{t_{1},\cdots,t_{d}\rangle}\to W^{*}(\ell)$  the unique $*$-homomorphism satisfying $\pi_{\ell}(t_{j})=x_{j}$.
% \end{rem}

% \begin{rem} \label{rem:lawGNS}
% In fact, the proof of (ii) $\implies$ (i) gives an explicit description of the von Neumann algebra through the Gelfand-Naimark-Segal construction (see \cite[Proposition 5.2.14]{AGZ2009}).  Let $H = L^2(\ell)$ be separation-completion of $\bC\ang{t_1,\dots,t_d}$ with respect to the semi-inner product $\langle{f,g\rangle}_\ell = \ell(f^*g)$.  One can show that multiplication by $t_j$ gives a well-defined, bounded, self-adjoint operator on $\mathcal{H}$.  We take $M$ to be the von Neumann algebra generated by $x_1$, \dots, $x_d$, and let $\tau$ be the state corresponding to the vector $1$ in $\mathcal{H}$.  In fact, we will denote $M$ by $\mathrm{W}^*(\ell)$ and we denote by $\pi_\ell$ the unital $*$-homomorphism $\bC\ang{t_1,\dots,t_d}$ sending $t_j$ to $x_j$.
% \end{rem}

%  \ref{prop:NClaws} (\ref{item:abstract characterization of law}).  %
\subsection{Microstate spaces and $1$-bounded entropy}

Let $(M,\tau)$ be a diffuse tracial von Neumann algebra, and $x\in M_{\sa}^d$ for some $d \in \bN$ with $W^{*}(x)=M$.  Suppose that $\norm{x}_\infty \leq R$.  Following \cite{VoiculescuFreeEntropy2}, for each open set $\mathcal{O}$ of $\Sigma_{d,R}$ and $N \in \bN$, we define
\[
\Gamma_R^{(n)}(\cO) = \{X \in \mathbb{M}_n(\bC)_{\sa}^d: \ell_X \in \mathcal{O} \}.
\]
When $\mathcal{O}$ is a neighborhood of $\ell_x$, we call $\Gamma_R^{(n)}(\mathcal{O})$ a \emph{microstate space} for $x$.

Given $d,n\in \bN$, $p\in [1,\infty]$, $\varepsilon>0$ and $\Omega,\Xi\subseteq \bM_{n}(\bC)^{d}$ then $\Xi$ is said to \emph{$(\varepsilon,\norm{\cdot}_p)$-cover $\Omega$} if for every $A\in \Omega$, there is a $B\in \Xi$ with $\|A-B\|_{p}<\varepsilon$. We define the \emph{covering number} of $\Omega\subseteq \bM_{n}(\bC)^{d}$, denote $K_{\varepsilon}(\Omega,\|\cdot\|_{p})$, to be the minimal cardinality of a set that $(\varepsilon,\|\cdot\|_{p})$-covers $\Omega$. %We will use covering numbers for different values of $p$ in Section \ref{sec: vanishing L2B}.
For subsets of $\bM_{n}(\bC)^{d}$ which are invariant under the conjugation action of $\mathcal{U}(n)$ on $\bM_{n}(\bC)^{d}$, it is natural to take the orbital numbers modulo unitary conjugation.
% While these covering numbers are natural for many purposes, for unitarily invariant subsets of matrices, it is natural to take the orbital numbers modulo unitary conjugation.
Given $n \in \bN$, $\varepsilon > 0$ and $\Omega,\Xi \subseteq \mathbb{M}_{n}(\bC)^d$ we say that $\Xi$ \emph{orbitally $(\varepsilon,\norm{\cdot}_2)$-covers} $\Omega$ if for every $A\in\Omega$, there is a $B \in \Xi$ and an $n \times n$ unitary matrix $V$ so that
\[
\norm{A-VBV^*}_2 < \varepsilon.
\]
We define the \emph{orbital covering number} $K_{\varepsilon}^{\textnormal{orb}}(\Omega,\norm{\cdot}_2)$ as the minimal cardinality of a set of $\Omega_0$ that orbitally $(\varepsilon,\norm{\cdot}_2)$-covers $\Omega$. Since we will usually be concerned with $\|\cdot\|_{2}$-norms we will frequently drop $\|\cdot\|_{2}$ from the notation and use $K_{\varepsilon}^{\textnormal{orb}}(\Omega)$ instead of $K_{\varepsilon}^{\textnormal{orb}}(\Omega,\|\cdot\|_{2})$.
Let $R\in [0,\infty)$ be such that $\norm{x}_\infty < R$.

% The $1$-bounded entropy is defined in terms of the exponential growth rate of covering numbers of $\Gamma_R^{(n)}(\mathcal{O})$ (up to unitary conjugation) as $n \to \infty$ for neighborhoods $\mathcal{O}$ of $\ell_x$.
% Given $n \in \bN$, $\varepsilon > 0$ and $\Omega,\Xi \subseteq \mathbb{M}_{n}(\bC)^d$ we say that $\Xi$ \emph{orbitally $(\varepsilon,\norm{\cdot}_2)$-covers} $\Omega$ if for every $A\in\Omega$, there is a $B \in \Xi$ and an $n \times n$ unitary matrix $V$ so that
% \[
% \norm{A-VBV^*}_2 < \varepsilon.
% \]
% We define the \emph{orbital covering number} $K_{\varepsilon}^{\textnormal{orb}}(\Omega,\norm{\cdot}_2)$ as the minimal cardinality of a set of $\Omega_0$ that orbitally $(\varepsilon,\norm{\cdot}_2)$-covers $\Omega$. Since we will usually be concerned with $\|\cdot\|_{2}$-norms we will frequently drop $\|\cdot\|_{2}$ from the notation and use $K_{\varespilon}^{\textnormal{orb}}(\Omega)$ instead of $K_{\varepsilon}^{\textnormal{orb}}(\Omega,\|\cdot\|_{2})$.
% Let $R\in [0,\infty)$ be such that $\norm{x}_\infty < R$.
For a weak$^{*}$-neighborhood $\mathcal{O}$ Of $\ell_{x}$, we define
\begin{align*}
h_{R,\varepsilon}(\mathcal{O})& := \limsup_{n \to\infty}\frac{1}{n^2}\log K_{\varepsilon}^{\orb}(\Gamma_{R}^{(n)}(\mathcal{O})), \\
h_{R,\varepsilon}(x) &:= \inf_{\mathcal{O} \ni \ell_x} h_{R,\varepsilon}(\mathcal{O}),
\end{align*}
where the infimum is over all weak$^{*}$-neighborhoods $\mathcal{O}$ of $\ell_{x}$.  We then define
\[
h_R(x) := \sup_{\varepsilon > 0} h_{R,\varepsilon}(x).
\]
By \cite{Hayes2018}, it follows that $h_{R}(x)$ is independent of $R$ as soon as $\norm{x}_{\infty}\leq R$, so we use $h(x)$ instead of $h_{R}(x)$ as soon as $\norm{x}_{\infty}\leq R$. Moreover, if $x,y$ are two self-adjoint tuples in $M$ with $W^{*}(x)=M=W^{*}(y)$, then $h(x)=h(y)$. So we may define $h(M)=h(x)$ if $W^{*}(x)=M$. If $M$ is not a factor, then the $1$-bounded entropy depends upon the choice of $M$. We will use $h(M,\tau)$ if we wish to emphasize the dependence of the $1$-bounded entropy of $\tau$. Usually the choice of $\tau$ will be clear from the context and use $h(M)$. In \cite{Hayes2018}, it's show how to extend this definition to infinitely many variables, but we will not need this.
The $1$-bounded entropy characterizes strong $1$-boundedness by the following result.
% In \cite{Hayes2018}, it is shown that $h_R(x)$ is independent of $R$ provided that $\norm{x}_\infty \leq R$, and hence we may unambiguously denote it by $h(x)$.
% In fact, $h(x)$ only depends on the von Neumann algebra generated by $x$.  Hence, for every finitely generated tracial von Neumann algebra $M$, we may define $h(M)$ as $h(x)$ for some generating tuple $x$.  If $M$ is not a factor, then the $1$-bounded entropy depends upon the trace we choose on $M$. We will use $h(M,\tau)$ for a tracial von Neumann algebra $(M,\tau)$ if we wish to emphasize the dependence of the $1$-bounded entropy on $\tau$. However, we will typically suppress this from the notation and just use $h(M)$ unless there is a possibility of confusion.  The definition of $h$ can be extended even to the case where $M$ is not finitely generated; see \cite{Hayes2018}.  However, our results only require a direct appeal to the definition in the finitely generated setting.
% The importance of  $1$-bounded entropy to the study of strong $1$-boundedness is encapsulated by the following result.

\begin{thm}[{See  \cite[Proposition A.16]{Hayes2018}}]
A tracial von Neumann algebra $M$ is strongly $1$-bounded in the sense of Jung \cite{Jung2007} if and only if $h(M) < \infty$.
\end{thm}\label{thm:S1B in terms of 1-bounded entropy}
Because of this, we will not use Jung's original definition of strongly $1$-bounded \cite{Jung2007} and will prove algebras are strongly $1$-bounded by showing they have finite $1$-bounded entropy.

\section{Proof of Theorem \ref{thm: von Neumann algebraic}} \label{sec: vanishing L2B}

\subsection{Sketch of the proof} \label{subsec: sketch}

In order to prove strong $1$-boundedness, or equivalently that $h_R(x) < \infty$, we will estimate $h_{R,\varepsilon}(x)$ iteratively for smaller and smaller values of $\varepsilon$ in a similar manner to Jung \cite{Jung2007}.  In particular, if $\eta \leq \varepsilon$, then we want to estimate $h_{R,\eta}(x)$ in terms of $h_{R,\varepsilon}(x)$ by covering a $(\varepsilon,\norm{\cdot}_2)$-ball in the microstate space by an $(\eta,\norm{\cdot}_2)$-balls.

Consider the $(\varepsilon,\norm{\cdot}_2)$-ball centered at some microstate $X \in \mathbb{M}_n(\bC)_{\sa}^d$ with $\norm{X}_\infty \leq R$.  Let $D_f(X)$ denote the matrix of tensors as in the theorem statement with $x$ replaced by $X$.  If $Y$ is a microstate in the $(\varepsilon,\norm{\cdot}_2)$-ball of $X$, then by Taylor expansion $f(Y) - f(X)$ is approximately $\partial f(X) \# (Y - X)$.  
By taking a high degree of approximation for our microstate space, we can make $f(Y) - f(X)$ arbitrarily small, and thus arrange that $Y - X$ is in the approximate kernel of $\partial f(X)$.  Furthermore, because we are only considering microstates up to unitary orbits, we can assume without loss of generality that $Y$ is the closest point in its unitary orbit to $X$, which implies that $\sum_{j=1}^m [X_j, Y_j] = 0$ (see Lemma \ref{lemm:orbitminimizer}).  Hence, $Y - X$ is in the approximate kernel of $D_f(X)$.  Because $\int |\log t|\,d\mu_{|D_f(x)|}(t) < \infty$ and $|D_f(X)|$ converges in distribution to $|D_f(x)|$, the dimension of the kernel of $D_f(X)$ vanishes in comparison to $n^2$, and we can use standard estimates on covering numbers of approximate kernels to get a bound on the $\eta$-covering number.

\subsection{Background on non-commutative derivatives and Taylor expansion} \label{subsec: DQ}

First, we recall Voiculescu's free difference quotient (\cite{VoiculescuFreeEntropy2, VoiculescuV}).  Consider the $d$-variable non-commutative polynomial algebra $\bC\ang{t_1,\dots,t_d}$.  Let $\partial_j: \bC\ang{t_1,\dots,t_d} \to \bC\ang{t_1,\dots,t_d} \otimes \bC\ang{t_1,\dots,t_d}$ be the unique linear map satisfying
\[
\partial_j[t_{i_1} \dots t_{i_k}] = \sum_{\alpha=1}^k \delta_{i_{\alpha} = j} t_{i_1} \dots t_{i_{\alpha-1}} \otimes t_{i_{\alpha+1}} \dots t_{i_k}.
\]
The map $\partial_j$ can also be characterized as the unique derivation $\bC\ang{t_1,\dots,t_d} \to \bC\ang{t_1,\dots,t_d} \otimes \bC\ang{t_1,\dots,t_d}$ satisfying $\partial_j(t_i) = \delta_{i=j} (1 \otimes 1)$.  Here, when we describe $\partial_j$ as a ``derivation,'' we are viewing $\bC\ang{t_1,\dots,t_d} \otimes \bC\ang{t_1,\dots,t_d}$ as a bimodule over $\bC\ang{t_1,\dots,t_d}$ using the multiplication operations
\[
p(f \otimes g) = pf \otimes g, \qquad (f \otimes g)p = f \otimes gp.
\]
If $f = (f_1,\dots,f_m) \in \bC\ang{t_1,\dots,t_d}^m$, then
\[
\partial f \in \mathbb{M}_{m,d}(\bC\ang{t_1,\dots,t_d} \otimes \bC\ang{t_1,\dots,t_d})
\]
will denote the matrix whose $(i,j)$ entry is $\partial_j f_i$.  This matrix plays a similar role to the derivative of a function $\bR^d \to \bC^m$, in that it furnishes the first-order term in a non-commutative Taylor expansion for the evaluation of $f$ on elements of a tracial von Neumann algebra.

Recall that if $(M,\tau)$ is a tracial von Neumann algebra and $f \in \bC\ang{t_1,\dots,t_d}$ and $x = (x_1,\dots,x_d) \in M_{\sa}^d$, then the evaluation of $f(x)$ is the image of $f$ under the unique unital $*$-homomorphism $\bC\ang{t_1,\dots,t_d} \to M$ given by $t_j \mapsto x_j$.  The evaluation of $f = (f_1,\dots,f_m)$ on $x = (x_1,\dots,x_d)$ is defined by $(f_1(x),\dots,f_m(x))$.  Moreover, $f, g \in \bC\ang{t_1,\dots,t_d}$, we set
\[
(f \otimes g)(x) = f(x) \otimes g(x)^{\op} \in M \otimes M^{\op},
\]
where $M^{\op}$ denotes the opposite algebra of $M$\footnote{$M^{\op}$ is an algebra with the same addition and $*$-operation but the order of multiplication is reversed; note that $M^{\op}$ is a tracial von Neumann algebra}, and $\otimes$ is the algebraic tensor product.  By extending this operation linearly, we can define $F(x) \in \mathbb{M}_{m,d}(M \otimes M^{\op})$
% for $F \in \bC\ang{t_1,\dots,t_d} \otimes \bC\ang{t_1,\dots,t_d}$, and in fact 
for $F \in \mathbb{M}_{m,d}(\bC\ang{t_1,\dots,t_d} \otimes \bC\ang{t_1,\dots,t_d})$.

For $a, b, x \in M$, we define
\[
(a \otimes b^{\op}) \# x = axb.
\]
This extends to a bilinear map $(M \otimes M^{\op}) \times M \to M$.  If $A \in \mathbb{M}_{m,d}(M \otimes M^{\op})$ and $x \in M_{\sa}^d$, we define $A \# x \in M^m$ as the vector with entries
\[
(A \# x)_i = \sum_{j=1}^d A_{i,j} \# x_j.
\]
The first-order Taylor approximation is as follows.  Note that in contrast with the classical Taylor approximation where the error estimates are typically given in the Euclidean norm or $2$-norm on $\bR^m$, we have to mix different non-commutative $p$-norms of $y - x$ in the estimates.

\begin{lemm} \label{lemm: Taylor approximation}
Let $f \in \bC\ang{t_1,\dots,t_d}^m$ and let $R > 0$.  Then there exists  constants $A_f$, $B_f$, $C_f$ depending only on $f$ and $R$, such that for every tracial von Neumann algebra $(M,\tau)$ and $x, y \in M_{\sa}^d$ with $\norm{x}_\infty, \norm{y}_\infty \leq R$, we have
\begin{align}
\norm{f(x)}_\infty &\leq A_f \label{eq:polynomialbounded}\\
\norm{f(y) - f(x)}_2 &\leq B_f \norm{y - x}_2 \label{eq:polynomialLipschitz} \\
\norm{f(y) - f(x) - \partial f(x) \# (y - x)}_1 &\leq C_f \norm{y - x}_2^2. \label{eq:Taylorexpansion}
\end{align}
\end{lemm}

\begin{proof}
The case of general $m$ will follow from applying the $m = 1$ case componentwise.  For the $m = 1$ case, to verify the claims for every non-commutative polynomial $f$, it suffices to check them for $f(t_1,\dots,t_d) = t_j$ and show that they are preserved under linear combinations and products.
\begin{enumerate}[(1)]
    \item For $f(t_1,\dots,t_d) = t_j$, the claims hold with $A_f = R$, $B_f = 1$, $C_f = 0$ since $\partial_i f = \delta_{i=j} (1 \otimes 1)$.
    \item If $f$ and $g$ satisfy the claims and $\alpha$, $\beta \in \bC$, then $\alpha f + \beta g$ satisfies the claims with $A_{\alpha f + \beta g} = |\alpha| A_f + |\beta| A_g$ and the same for the $B$'s and $C$'s.
    \item Suppose $f, g \in \bC\ang{t_1,\dots,t_d}$ satisfy the conclusions of the lemma.  Then $fg$ satisfies \eqref{eq:polynomialbounded} with $A_{fg} = A_f A_g$.  Moreover, by writing $(fg)(y) - (fg)(y) = (f(y) - f(x)) g(y) + f(x)(g(y) - g(x))$ and using the $L^2$-$L^\infty$-H\"older inequality, $fg$ satisfies \eqref{eq:polynomialLipschitz} with $B_{fg} = B_f A_g + A_f B_g$.  Similarly, using algebraic manipulations and the fact that $\partial_j$ is a derivation,
    \begin{align*}
    (fg)(y) - (fg)(x) - \partial (fg)(x) \# (y - x) &= [f(y) - f(x) - \partial f(x) \# (y - x)] g(x) \\
    & \quad + f(x) [g(y) - g(x) - \partial g(x) \# (y - x)] \\
    & \quad + (f(y) - f(x))(g(y) - g(x)).
    \end{align*}
    We estimate the first two terms by the $L^1$-$L^\infty$ H\"older inequality and the third term by the $L^2$-$L^2$ H\"older inequality and thus obtain that $fg$ satisfies \eqref{eq:Taylorexpansion} with $C_{fg} = C_f A_g + C_g A_f + B_f B_g$.  \qedhere
\end{enumerate}
\end{proof}

The following lemma will be needed to show that the spectral measures of certain operators on $\mathbb{M}_n(\bC)^d$ associated to matricial microstates for $x \in M_{\sa}^d$ converge as $n \to \infty$ to the spectral measures of corresponding operators from a tracial von Neumann algebra.  In the following, for a tracial von Neumann algebra $M$, we denote by $M \overline{\otimes} M^{\op}$ the tracial von Neumann algebraic tensor product of $M$, equipped with the trace $\tau_M \otimes \tau_{M^{\op}}$.  If $M \overline{\otimes} M^{\op}$ is represented on the Hilbert space $H$, then $\mathbb{M}_{m,d}(M \overline{\otimes} M^{\op})$ are represented as operators $H^d \to H^m$.  Also, $\mathbb{M}_d(M \overline{\otimes} M^{\op})$ is a tracial von Neumann algebra and can be equipped with the normalized trace $\tr_d \otimes \tau_M \otimes \tau_{M^{\op}}$ where $\tr_d$ is the normalized trace on $\mathbb{M}_d(\bC)$.  Moreover, $\mathcal{P}(\bR)$ denotes the space of probability measures on $\bR$ equipped with the weak$^*$ topology as linear functionals on $C_0(\bR)$.

\begin{lemm} \label{lemm:weak*-weak* continuity}
Let $d, m\in \bN$, $f \in \bC\ang{t_1,\dots,t_d}^m$, and $R > 0$.  For $\ell \in \Sigma_{d,R},$ let $\pi_{\ell}\colon \bC\ang{t_1,\cdots,t_d}\to W^{*}(\ell)$ be the GNS construction corresponding to $\ell$ as in Remark after Proposition \ref{prop:NClaws}.  Let
\[
F \in \mathbb{M}_{m,d}(\bC\ang{t_1,\dots,t_d} \otimes \bC\ang{t_1,\dots,t_d}),
\]
consider $F(\pi_\ell(t_1,\dots,t_d)) \in \mathbb{M}_{m,d}(\mathrm{W}^*(\ell) \otimes \mathrm{W}^*(\ell)^{\op})$, and let $\mu_{|F(\pi_\ell(t_1,\dots,t_d))|}$ be the spectral measure of $|F(\pi_\ell(t_1),\dots,\pi_\ell(t_d))| = (F(\pi_\ell(t_1),\dots,\pi_\ell(t_d))^*F(\pi_\ell(t_1),\dots,\pi_\ell(t_d)))^{1/2}$ as an element of $\mathbb{M}_d(\mathrm{W}^*(\ell) \overline{\otimes} \mathrm{W}^*(\ell)^{\op})$.  Then the map $\Sigma_{d,R} \to \mathcal{P}(\bR): \ell \mapsto \mu_{|F(\pi_\ell(t_1,\dots,t_d))|}$ is weak$^{*}$-weak$^{*}$ continuous.
\end{lemm}

\begin{proof}
Because $F(x)$ is a linear combination of simple tensors of polynomials, there is some universal constant $C$ depending on $F$ and $R$ such that $\norm{F(x)}_{\mathbb{M}_{m,d}(M \overline{\otimes} M^{\op})} \leq K$ for every tuple of operators with $\norm{x}_\infty \leq R$.  In particular, the spectral measure of $|F(x)|$ is supported on $[0,K]$.  Hence, it suffices to show that for every $\phi \in C([0,K])$, the map
\[
\ell \mapsto (\tr_d \otimes \tau_{\mathrm{W}^*(\ell)} \otimes \tau_{\mathrm{W}^*(\ell)}^{\op})(\phi(|F(\pi_\ell(t_1),\dots,\pi_\ell(t_d))|))
\]
is continuous.  By the Stone-Weierstrass theorem, it suffices to consider the case when $\phi(s) = \psi(s^2)$ where $\psi$ is a polynomial.  In this case,
\[
\phi(|F(\pi_\ell(t_1),\dots,\pi_\ell(t_d))|)) = \psi(F(\pi_\ell(t_1),\dots,\pi_\ell(t_d))^* F(\pi_\ell(t_1),\dots,\pi_\ell(t_d)).
\]
The right-hand side is just an element of $\mathbb{M}_d(\bC\ang{t_1,\dots,t_d} \otimes \bC\ang{t_1,\dots,t_d})$ applied to the operators $\pi_\ell(t_1)$, \dots, $\pi_\ell(t_d)$.  Hence,
\[
(\tr_d \otimes \tau_{\mathrm{W}^*(\ell)} \otimes \tau_{\mathrm{W}^*(\ell)}^{\op})(\phi(|F(\pi_\ell(t_1),\dots,\pi_\ell(t_d))|)) = (\tau_{\mathrm{W}^*(\ell)} \otimes \tau_{\mathrm{W}^*(\ell)}^{\op})(G(\pi_\ell(t_1),\dots,\pi_\ell(t_d))),
\]
where $G \in \bC\ang{t_1,\dots,t_d} \otimes \bC \ang{t_1,\dots,t_d}$ is $1/d$ times the sum of the diagonal entries of this matrix of tensors of polynomials.  Since $G$ is a linear combination of simple tensors, it suffices to show the continuity of the map
\[
\ell \mapsto (\tau_{\mathrm{W}^*(\ell)} \otimes \tau_{\mathrm{W}^*(\ell)}^{\op})((f \otimes g)(\pi_\ell(t_1),\dots,\pi_\ell(t_d))),
\]
where $f, g \in \bC \ang{t_1,\dots,t_d}$.  But the right-hand side is equal to $\ell(f) \ell(g)$, and $\ell \mapsto \ell(f) \ell(g)$ is continuous by definition of the weak$^*$ topology.
\end{proof}

\subsection{Covering the microstate space}

We now give the details of the argument sketched in \S \ref{subsec: sketch}.  We begin with the orbital optimization trick.  This lemma also appears in \cite[Lemma 1.14]{GJNS2021}, where it is related with non-commutative optimal transport theory.

\begin{lemm} \label{lemm:orbitminimizer}
Let $X$, $Y \in \mathbb{M}_n(\bC)_{\sa}^d$.  There exists a unitary matrix $U$ that minimizes $\norm{X - UYU^{*}}_2$, and any such unitary satisfies $\sum_{j=1}^d [X_j, UY_jU^{*}] = 0$.
\end{lemm}

\begin{proof}
A minimizer exists because the unitary group is compact and the function $U \mapsto \norm{X - UYU^*}_2$ is continuous.  Suppose $U$ is a minimizer and let $A \in \mathbb{M}_n(\bC)_{\sa}$.  Then
\begin{align*}
0 &\leq \norm{X - e^{itA}UYU^*e^{-itA}}_2^2 - \norm{X - UYU}_2^2 \\
&= 2 \ang{X, e^{itA}UYU^*e^{-itA} - UYU^*}.
\end{align*}
Differentiating at $t = 0$, we get
\[
0 = \sum_{j=1}^d \tr_n(X_j i[A,UY_jU^*]) = \sum_{j=1}^d \tr_n(i[UY_jU^*,X_j]A).
\]
Because $A$ was arbitrary, we have $\sum_{j=1}^d [UY_jU^*,X_j] = 0$.
\end{proof}

Next, we will give an initial form of the iterative estimate in terms of an auxiliary quantity $\Psi_{R,\eta,\delta,\varepsilon}(x,f)$ measuring the size of approximate kernels of $D_f(X)$.  For a neighborhood $\mathcal{O}$ of $\ell_x$ in $\Sigma_{d,R}$, define
\[
\Psi_{R,\eta,\delta,\varepsilon}(\mathcal{O},f)
= \limsup_{n \to \infty} \frac{1}{n^2} \sup_{X \in \Gamma_R^{(n)}(\mathcal{O})} \log K_\eta \left( \left\{Z: \norm{Z} \leq 2R, ~ \norm{Z}_2 < \delta, ~ \norm{D_f(X) \# Z}_1 < \varepsilon \right\}, \right).
\]
Note that $\Psi_{R,\eta,\delta,\varepsilon}(\mathcal{O},f)$ is monotone in $\mathcal{O}$.  We define
\[
\Psi_{R,\eta,\delta,\varepsilon}(x,f) = \inf_{\mathcal{O}} \Psi_{R,\eta,\delta,\varepsilon}(\mathcal{O},f)
\]
At this point, the reader may be wondering why we use $\norm{D_f(X) \# Z}_1 < \varepsilon$ instead of $\norm{D_f(X) \# Z}_2 < \varepsilon$.  The reason is that the error estimate in the non-commutative Taylor expansion requires the $1$-norm rather than the $2$-norm, that is, $\norm{f(Y) - f(X) - \partial f(X) \# (Y - X)}_1 \leq C \norm{Y - X}_2^2$.  Later we will work to estimate this in terms of the approximate kernel with the error measured in $2$-norm.

\begin{lemm} \label{lemm: iterative kernel estimate}
With the set up of Theorem \ref{thm: von Neumann algebraic}, there is a constant $C > 0$ (depending only upon $f$ and $R$) so that for all $\varepsilon,\eta > 0$ we have that
\[
h_{R,\eta}(x) \leq h_{R,\varepsilon}(x) + \Psi_{2R,\eta/2,2\varepsilon,C\varepsilon^2}(x,f).
\]
\end{lemm}

\begin{proof}
Fix the neighborhood
\[
\mathcal{U} = \left\{\ell: \sum_{j=1}^m \ell(f_j^*f_j)^{1/2} < \varepsilon^2 \right\} \subseteq \Sigma_{d,R}.
\]
In order to estimate $h_{R,\eta}(x,\norm{\cdot}_2)$, pick a neighborhood $\mathcal{O}$ of $\ell_x$, and then we will cover the microstate space $\Gamma^{(n)}(\mathcal{O} \cap \mathcal{U})$ by orbital $(\eta,\norm{\cdot}_2)$-balls.  Recall that if a set can be covered by a certain number of $\varepsilon$-balls with centers not necessarily in that set, then it can be covered with the same number of $2\varepsilon$-balls with centers in the set.  Hence, there exists a set $\Omega \subseteq \Gamma_R^{(n)}(\cO \cap \cU)$ of cardinality at most $K_{\varepsilon}(\Gamma_R^{(n)}(\cO \cap \cU),\norm{\cdot}_2)$ such that the $(2\varepsilon,\norm{\cdot}_2)$-balls centered at $X$ in $\Omega$ cover $\Gamma_R^{(n)}(\cO \cap \cU)$.

We want to cover each of the orbital $(\varepsilon,\norm{\cdot}_2)$-balls  by orbital $(\eta,\norm{\cdot}_2)$-balls.  If $Y$ is in the orbital $(2\varepsilon,\norm{\cdot}_2)$-ball around $X$, then because we only need to cover $Y$ up to unitary equivalence, we can assume without loss of generality that $Y$ is the element of its orbit that is closest to $X$ in $\norm{\cdot}_2$, and thus $\sum_j [X_j, Y_j] = 0$ by Lemma \ref{lemm:orbitminimizer}.  Recall by Lemma \ref{lemm: Taylor approximation},
\[
f(Y) - f(X) = \partial f(X) \# (Y - X) + \Delta_f(X,Y),
\]
where the error term $\Delta_f(X,Y)$ satisfies
\[
\norm{\Delta_f(X,Y)}_1 \leq C_f \norm{Y - X}_2^2 \leq 4C_f \varepsilon^2
\]
for a constant $C_f$ depending only on $f$ and $R$.  By our choice of $\mathcal{U}$, we have
\[
\norm{f(X)}_1 \leq \sum_{j=1}^\infty \tr_n(f_j(X)^* f_j(X))^{1/2} < \varepsilon^2,
\]
and similarly $\norm{f(Y)}_1 < \varepsilon^2$.  It follows that
\[
\norm{\partial f(X) \# (Y - X)}_1 < (2 + 4C_f) \varepsilon^2.
\]
Let $Z = Y - X$.  Note that $\sum_{j=1}^d [X_j, Z_j] = \sum_{j=1}^d [X_j, Y_j] = 0$. Since $\sum_{j=1}^{d}[X_{j},Z_{j}]=0$, we have:
\[\|D_{f}(X)\#(Y-X)\|_{1}=\|(\partial f)(X)\#(Y-X)\|_{1}<(2+4C_{f})\varepsilon^{2}.\]
Also, 
$\norm{Z}_\infty \leq 2R$.  Of course, the number of $(\eta,\norm{\cdot}_2)$-balls needed to cover the set of $Z$'s obtained in this way is at most
\[
\sup_{X \in \Gamma_R^{(n)}(\mathcal{O})} K_{\eta} \left( \left\{Z: \norm{Z} \leq 2R, ~\norm{Z}_2 < 2 \varepsilon, ~ \norm{D_f(X) \# Z}_1 < (2 + 4C_f)\varepsilon^2 \right\}, \norm{\cdot}_2 \right).
\]
It follows that
\begin{multline*}
K_{\eta}^{\orb}(\Gamma_R^{(n)}(\mathcal{O} \cap \mathcal{U}),\norm{\cdot}_2) \leq \frac{1}{N^2} K_{\varepsilon}^{\orb}(\Gamma_R^{(n)}(\mathcal{O} \cap \mathcal{U}),\norm{\cdot}_2) \\ \sup_{X \in \Gamma_R^{(n)}(\mathcal{O})} K_{\eta/2} \left( \left\{Z: \norm{Z} \leq 2R, ~ \sum_{j=1}^d [Z_j,X_j] = 0, ~ \norm{Z}_2 < 2\varepsilon, ~ \norm{\partial f(X) \# Z}_1 < (2 + C_f)\varepsilon^2 \right\}, \norm{\cdot}_2 \right).
\end{multline*}
Apply $\limsup_{n \to \infty} (1/n^2) \log$ to obtain
\[
h_{R,\eta}(\mathcal{O} \cap \mathcal{U},\norm{\cdot}_2) \leq h_{R,\varepsilon}(\mathcal{O} \cap \mathcal{U}) + \Psi_{2R,\eta/2,2\varepsilon,(2+C_f) \varepsilon^2}(\mathcal{O} \cap \mathcal{U},f).
\]
Because all the covering numbers are monotone in the ``$\mathcal{O}$'' variable, taking the infimum over all $\mathcal{O}$ yields the same result whether or not we intersect with $\mathcal{U}$ first.  Thus, upon taking the infimum with respect to $\mathcal{O}$, we obtain the asserted result.
\end{proof}

\subsection{Covering the approximate kernel}

In order to convert our estimate with the $\norm{\cdot}_1$-approximate kernel to an estimate with the $\norm{\cdot}_2$-approximate kernel, we will estimate in Lemma \ref{lemm: norm switch} the $\norm{\cdot}_2$-covering number of the intersection of a $\norm{\cdot}_1$-ball and a $\norm{\cdot}_\infty$ ball.  We employ Szarek's covering estimate in a similar way to \cite{PropTS1B}. For convenience of the reader, we state the Lemma explicitly here.

\begin{lemm}[\cite{PropTS1B}] \label{lemm: Szarek}
There exists a universal constant $C$ such that for $t \geq 0$,
\begin{align*}
K_{\varepsilon}(\{P \in \mathbb{M}_n(\bC) \text{ projection, } \tr_n(P) \leq t\},\|\cdot\|_{\infty}) &= K_{\varepsilon}(\{P \in \mathbb{M}_n(\bC) \text{ projection, } \tr_n(P) \geq 1 - t \}, \norm{\cdot}_\infty) \\
&\leq (1 + nt) \left( \frac{C}{\varepsilon} \right)^{2n^2t}.
\end{align*}
\end{lemm}
% to Lemma \ref{lemm: property T iterative estimate}.

\begin{lemm} \label{lemm: norm switch}
There is a universal constant $C$ such that for $t > 0$ and $\varepsilon \leq 3R$,
\[
\limsup_{n \to \infty} \frac{1}{n^2} \log K_{\varepsilon}(B_{\mathbb{M}_n(\bC),\norm{\cdot}_\infty}(0,R) \cap B_{\mathbb{M}_n(\bC),\norm{\cdot}_1}(0,t \varepsilon), \norm{\cdot}_\infty) \leq 12 t \log \frac{CR}{\varepsilon}.
\]
\end{lemm}

\begin{proof}
By Lemma \ref{lemm: Szarek}, there exists a set $\Xi$ of projections of rank at least $n(1 - 3t)$ such that every projection $P$ of rank at least $n(1 - 3t)$ satisfies $\norm{P - Q}_\infty < \varepsilon / 3R$ for some $Q \in \Xi$ and such that
\[
|\Xi| \leq (1 + nt) \left( \frac{6C_1R}{\varepsilon} \right)^{6n^2t}.
\]
Next, for each $Q \in \Xi$, observe that $(1 - Q)\mathbb{M}_n(\bC)_{\sa}$ is a Hilbert space of real dimension at most $6n^2 t$, and hence for some constant $C_2$,
\[
K_{\varepsilon/3}((1 - Q)B_{\mathbb{M}_n(\bC),\norm{\cdot}_\infty}(0,R),\norm{\cdot}_\infty) \leq \left( \frac{3C_2R}{\varepsilon} \right)^{6n^2t}.
\]
Therefore, we may choose a set $\Omega_Q$ with $|\Omega_Q| \leq \left( \frac{3C_2R}{\varepsilon} \right)^{2n^2 t}$ that $(\varepsilon/3,\norm{\cdot}_\infty)$-covers $(1 - Q)B_{\mathbb{M}_n(\bC),\norm{\cdot}_\infty}(0,R)$.

We claim $\Omega = \bigcup_{Q \in \Xi} \Omega_Q$ is an $(\varepsilon,\norm{\cdot}_\infty)$-covering of $B_{\mathbb{M}_n(\bC),\norm{\cdot}_\infty}(0,R) \cap B_{\mathbb{M}_n(\bC),\norm{\cdot}_1}(0,\delta \varepsilon/3)$.  Let $A \in B_{\mathbb{M}_n(\bC),\norm{\cdot}_\infty}(0,R) \cap B_{\mathbb{M}_n(\bC),\norm{\cdot}_1}(0,t\varepsilon)$, and let $\mu_{|A|}$ be the spectral measure of $|A|$, which is supported on $[0,R]$.  Let $P = 1_{[0,\varepsilon/3)}(|A|)$.  Note that
\[
\mu_{|A|}([\varepsilon/3,\infty)) \leq \frac{3}{\varepsilon} \int_\varepsilon^\infty x\,d\mu_{|A|}(x) \leq \frac{3\norm{A}_1}{\varepsilon} \leq 3t.
\]
Therefore,
\[
\rank(P) = n\mu_{|A|}([0,\varepsilon/3)) \geq n(1 - 3t).
\]
Choose $Q \in \Xi$ such that $\norm{P - Q}_\infty < \varepsilon/3R$.  There is some $B \in \Omega_Q$ such that $\norm{B - QA}_\infty < \varepsilon/3$.  Observe that
\begin{align*}
\norm{A - B} &\leq \norm{PA}_\infty + \norm{(P-Q)A}_\infty + \norm{QA - B}_\infty \\
&< \frac{\varepsilon}{3} + \frac{\varepsilon}{3R} R + \frac{\varepsilon}{3} \\
&= \varepsilon.
\end{align*}
It follows that
\[
K_{\varepsilon}(B_{\mathbb{M}_n(\bC),\norm{\cdot}_\infty}(0,R) \cap B_{\mathbb{M}_n(\bC),\norm{\cdot}_1}(0,t \varepsilon)), \norm{\cdot}_\infty) \leq (1 + 3nt) \left(\frac{6C_1R}{\varepsilon}\right)^{6n^2 t} \left( \frac{3C_2R}{\varepsilon} \right)^{6n^2t}.
\]
Let $C = \max(6C_1,3C_2)$.  Then
\[
\frac{1}{n^2} \log K_{\varepsilon}(B_{\mathbb{M}_n(\bC),\norm{\cdot}_\infty}(0,R) \cap B_{\mathbb{M}_n(\bC),\norm{\cdot}_1}(0,t \varepsilon)) \leq \frac{1}{n^2} \log (1 + 3nt) + 12 t \log \frac{CR}{\varepsilon}.
\]
Taking $n \to \infty$, we obtain the desired estimate.
\end{proof}

The second ingredient for estimating $\Psi_{R,\eta,\delta,\varepsilon}(x,f)$ is the following standard estimate for covering numbers of approximate kernels of operators on a Hilbert space.  Of course, we will apply this lemma to the operator $D_f(X) \#$ from the Hilbert space $\mathbb{M}_n(\bC)^d$ with the normalized Hilbert-Schmidt norm $\norm{\cdot}_2$ to the Hilbert space $\mathbb{M}_n(\bC)^m$ with $\norm{\cdot}_2$.  We remark that $B(\mathbb{M}_n(\bC)^d)$ is isomorphic to $\mathbb{M}_d(\mathbb{M}_n(\bC) \otimes \mathbb{M}_n(\bC)^{\op})$ acting on $\mathbb{M}_n(\bC)^d$ with the $\#$ operation, and the normalized trace on $B(\mathbb{M}_n(\bC)^d)$ corresponds to $\tr_d \otimes \tau_{\mathbb{M}_n(\bC)} \otimes \tau_{\mathbb{M}(\bC)^{\op}}$.

\begin{lemm}\label{lemm: packing approx kernel}
There is a universal constant $C > 0$ with the following property.  Let $\mathcal{H},\mathcal{K}$ be (complex) Hilbert spaces with $\mathcal{H}$ finite-dimensional, and let $T\in B(\mathcal{H},\mathcal{K})$. Fix $R > 0.$ For any $\delta, \varepsilon, \eta > 0$ with $\eta < 3$ we have that,
\[
K_\eta(\{\xi\in \mathcal{H}: \norm{\xi} < \delta, \norm{T \xi} < \varepsilon)\leq \left(\frac{C\delta}{\eta}\right)^{2\dim(\mathcal{H})\mu_{|T|}\left(\left[0,\frac{2\varepsilon}{\eta}\right]\right)}.
\]
% \[K_{3\sqrt{\varepsilon}}(\{\xi\in \mathcal{H}:\|T\xi\|<\varepsilon,\|\xi\|\leq R\},\|\cdot\|_{2})\leq \left(\frac{CR}{\varepsilon}\right)^{\dim_{\mathcal{H}}\mu_{|T|}([0,\sqrt{\varepsilon}])}.\]
Where $\mu_{|T|}$ is the spectral measure of $|T|$ with respect to the normalized trace on $B(\cH)$.
\end{lemm}

\begin{proof}
Let $P = 1_{\left[0,\frac{2\varepsilon}{\eta}\right]}(|T|).$ Suppose $\xi \in \mathcal{H}$ and $\norm{T\xi} < \varepsilon$.  Then, by functional calculus,
\[
\norm{\xi - P\xi} = \norm{1_{(\frac{2\varepsilon}{\eta},\infty)}(|T|)\xi} \leq \frac{\eta}{2\varepsilon} \norm{T \xi} < \frac{\eta}{2}.
\]
Thus, $\{\xi\in \mathcal{H}: \norm{\xi} < \delta, \norm{T\xi} < \varepsilon\}$ is contained in the $\eta/2$-neighborhood of  $B_{P(\cH)}(0,\delta)$.  Thus,
\[
K_{\eta}(\{\xi\in \mathcal{H}: \norm{\xi} < \delta, \norm{T\xi} <\varepsilon\}) \leq K_{\eta/2}(B_{P(\cH)}(0,\delta)) \leq \left( \frac{C\delta}{\eta} \right)^{2\dim(P \cH)},
\]
since the real dimension of $P \cH$ is twice the complex dimension.  Then note that
\[
\dim(P \cH) = \tr_{\dim(\cH)}(P) \dim(\cH) = \mu_{|T|}([0,2\varepsilon/\eta]) \dim(\cH).   \qedhere
\]
\end{proof}

\begin{lemm} \label{lemm: psi estimate}
Let $t \in (0,1/3]$ and suppose that $Rt \geq \varepsilon$.  Then for some constants $C_2$ and $C_3$ depending on $f$, we have
\[
\Psi_{R,\eta,\delta,\varepsilon}(x,f) \leq \mu_{|D_f(x)|}([0,2\varepsilon/t\eta]) \log \frac{C_2 \delta }{\eta} + 12mt \log \frac{C_3Rmt}{\varepsilon}
\]
\end{lemm}

\begin{proof}
Let $\cO$ be a neighborhood of $\ell_x$ and $X \in \Gamma^{(n)}(\mathcal{O})$.  We want to estimate the $(\eta,\norm{\cdot}_2)$ covering number of
\[
B_{\mathbb{M}_n(\bC)_{\sa}^d,\norm{\cdot}_\infty}(0,R) \cap B_{\mathbb{M}_n(\bC)_{\sa}^d,\norm{\cdot}_2}(0,\delta) \cap D_f(X)^{-1}(B_{\mathbb{M}_n(\bC)^m,\norm{\cdot}_1}(0,\varepsilon)).
\]
There exists a constant $C_1$ depending on $f$ such that
\[
\norm{D_f(X) \# Z}_\infty \leq C_1 \norm{Z}_\infty,
\]
and in particular, this is bounded by $2C_1R$ when $\norm{Z}_\infty \leq 2R$.  Hence, it suffices to estimate the $(\eta,\norm{\cdot}_2)$-covering number of
\[
B_{\mathbb{M}_n(\bC)^d,\norm{\cdot}_2}(0,\delta) \cap D_f(X)^{-1}\left( B_{\mathbb{M}_n(\bC)^m,\norm{\cdot}_\infty}(0,2C_1R) \cap  B_{\mathbb{M}_n(\bC)^m,\norm{\cdot}_1}(0,\varepsilon) \right),
\]
where we use $D_f(X)$ to denote the linear transformation $D_f(X) \#: \mathbb{M}_n(\bC)^d \to \mathbb{M}_n(\bC)^m$.  Fix a set $\Omega \subseteq \mathbb{M}_n(\bC)^m$ that $(\frac{\varepsilon}{2t},\norm{\cdot}_2)$-covers $B_{\mathbb{M}_n(\bC)^m,\norm{\cdot}_\infty}(0,2C_1R) \cap  B_{\mathbb{M}_n(\bC)^m,\norm{\cdot}_1}(0,\varepsilon)$ and satisfies
\begin{align*}
|\Omega| &\leq K_{\varepsilon/2t}(B_{\mathbb{M}_n(\bC)^m,\norm{\cdot}_\infty}(0,2C_1R) \cap  B_{\mathbb{M}_n(\bC)^m,\norm{\cdot}_1}(0,\varepsilon), \norm{\cdot}_2) \\
&\leq K_{\varepsilon/2mt}(B_{\mathbb{M}_n(\bC),\norm{\cdot}_\infty}(0,2C_1R) \cap  B_{\mathbb{M}_n(\bC),\norm{\cdot}_1}(0,\varepsilon), \norm{\cdot}_\infty)^m,
\end{align*}
where for the last several steps we used that $\norm{\cdot}_2 \leq m \norm{\cdot}_\infty$ on $\mathbb{M}_n(\bC)_{\sa}^m$ and that $B_{\mathbb{M}_n(\bC)^m,\norm{\cdot}_\infty}(0,2C_1R) \cap  B_{\mathbb{M}_n(\bC)^m,\norm{\cdot}_1}(0,\varepsilon)$ is contained in the product of $m$ copies of $B_{\mathbb{M}_n(\bC),\norm{\cdot}_\infty}(0,2C_1R) \cap  B_{\mathbb{M}_n(\bC),\norm{\cdot}_1}(0,\varepsilon)$.  Then
\begin{multline*}
B_{\mathbb{M}_n(\bC)^d,\norm{\cdot}_2}(0,\delta) \cap D_f(X)^{-1}\left(B_{\mathbb{M}_n(\bC)^m,\norm{\cdot}_\infty}(0,C_1R) \cap  B_{\mathbb{M}_n(\bC)^m,\norm{\cdot}_1}(0,\varepsilon) \right) \\
 \subseteq \bigcup_{Y \in \Omega} B_{\mathbb{M}_n(\bC)^d,\norm{\cdot}_2}(0,\delta) \cap D_f(X)^{-1}(B_{\mathbb{M}_n(\bC)^m,\norm{\cdot}_2}(Y,\tfrac{\varepsilon}{2t})).
\end{multline*}
For each $Y \in \Omega$, if $B_{\mathbb{M}_n(\bC)^d,\norm{\cdot}_2}(0,\delta) \cap D_f(X)^{-1}(B_{\mathbb{M}_n(\bC)^m,\norm{\cdot}_2}(Y,\varepsilon/2t))$ is non-empty, then pick some $Z_Y$ in this set, so that
\begin{align*}
B_{\mathbb{M}_n(\bC)^d,\norm{\cdot}_2}(0,\delta) &\cap D_f(X)^{-1}(B_{\mathbb{M}_n(\bC)^m,\norm{\cdot}_2}(Y,\frac{\varepsilon}{2t})) \\
&\subseteq B_{\mathbb{M}_n(\bC)^d,\norm{\cdot}_2}(0,\delta) \cap D_f(X)^{-1}(B_{\mathbb{M}_n(\bC)^m,\norm{\cdot}_2}(D_f(X) \# Z_Y,\tfrac{\varepsilon}{2t})) \\
&\subseteq Z_Y + \left( B_{\mathbb{M}_n(\bC)^d,\norm{\cdot}_2}(0,2\delta) \cap D_f(X)^{-1}(B_{\mathbb{M}_n(\bC)^m,\norm{\cdot}_2}(0,\tfrac{\varepsilon}{t})) \right).
\end{align*}
By Lemma \ref{lemm: packing approx kernel},
\[
K_{\eta}\left(B_{\mathbb{M}_n(\bC)^d,\norm{\cdot}_2}(0,2\delta) \cap D_f(X)^{-1}(B_{\mathbb{M}_n(\bC)^m,\norm{\cdot}_2}(0,\tfrac{\varepsilon}{t})), \norm{\cdot}_2 \right) \leq \left( \frac{C_2 \delta}{\eta} \right)^{2dn^2 \mu_{|D_f(X)|}([0,2\varepsilon/t\eta])}.
\]
In particular,
\begin{multline*}
\frac{1}{n^2} \log K_{\eta}\Bigl(B_{\mathbb{M}_n(\bC)_{\sa}^d,\norm{\cdot}_\infty}(0,R) \cap B_{\mathbb{M}_n(\bC)_{\sa}^d,\norm{\cdot}_2}(0,\delta) \cap D_f(X)^{-1}(B_{\mathbb{M}_n(\bC)^m,\norm{\cdot}_1}(0,\varepsilon)), \norm{\cdot}_2 \Bigr) \\  \leq 2d \left( \sup_{X \in \Gamma^{(n)}(\cU)} \mu_{|D_f(X)|}([0,2\varepsilon/t\eta]) \right) \log \frac{C_2 \delta}{\eta} + \frac{m}{n^2} \log K_{\varepsilon/2mt}(B_{\mathbb{M}_n(\bC),\norm{\cdot}_\infty}(0,2C_1R) \cap  B_{\mathbb{M}_n(\bC),\norm{\cdot}_1}(0,\varepsilon), \norm{\cdot}_\infty)
\end{multline*}
By Lemma \ref{lemm: norm switch},
\[
\limsup_{n \to \infty} \frac{1}{n^2} \log K_{\varepsilon/2mt}(B_{\mathbb{M}_n(\bC),\norm{\cdot}_\infty}(0,2C_1R) \cap  B_{\mathbb{M}_n(\bC),\norm{\cdot}_1}(0,\varepsilon), \norm{\cdot}_\infty) \leq 12t \log \frac{C_3Rmt}{\varepsilon}
\]
Now observe that as $\mathcal{O}$ shrinks to $\{\ell_x\}$, the measures $\mu_{|D_f(X)|}$ for $X \in \Gamma^{(n)}(\mathcal{O})$ converge uniformly in distribution to $\mu_{|D_f(x)|}$ using Lemma \ref{lemm:weak*-weak* continuity}.  Thus, we have
\[
\limsup_{n \to \infty} \sup_{X \in \Gamma^{(n)}(\cU)} \mu_{|D_f(X)|}([0,2\varepsilon/t\eta]) \leq \mu_{D_f(x)}([0,2\varepsilon/t\eta]).
\]
Thus, when we take the $\limsup$ as $n \to \infty$, we obtain the assertion of the theorem.
\end{proof}

\subsection{Iteration of the estimates}

By combining Lemmas \ref{lemm: iterative kernel estimate} and \ref{lemm: psi estimate}, we obtain the following bounds.

\begin{cor}
Let $t \in (0,1/3]$ and $\eta \leq \varepsilon$ and $Rt \geq \varepsilon$.  Then
\begin{equation} \label{eq: shrinking estimate}
h_{R,\eta}(x) \leq h_{R,\varepsilon}(x) + \mu_{|D_f(x)|}\left( \left[0,\frac{C_1\varepsilon^2}{t\eta} \right] \right) \log \frac{C_2 \varepsilon}{\eta} + 12mt \log \frac{C_3Rmt}{\varepsilon^2}.
\end{equation}
In particular, if $\varepsilon$ is sufficiently small (depending on $R$ and $f$), we can take $\eta = \varepsilon^{4/3}$ and $t = \varepsilon^{1/3}$ to get
\begin{equation} \label{eq: iterative estimate}
h_{R,\varepsilon^{4/3}}(x,\norm{\cdot}_2) \leq h_{R,\varepsilon}(x,\norm{\cdot}_2) + \mu_{|D_{f}(x)|}([0,C_1 \varepsilon^{1/3}]) \log (C_2 \varepsilon^{-1/3}) + 12 m \varepsilon^{1/3} \log (C_3Rm \varepsilon^{-5/3}).
\end{equation}
\end{cor}

\begin{proof}[Proof of Theorem \ref{thm: von Neumann algebraic}]
Fix some $\varepsilon$ sufficiently small that we can apply \eqref{eq: iterative estimate}.  By repeated application of that estimate,
\begin{multline*}
h_{R,\varepsilon^{4^k/3^k}}(x) \\
\leq h_{R,\varepsilon}(x) + \sum_{j=0}^{k-1} \left( \mu_{|D_f(x)|}([0,C_1 \varepsilon^{4^j/3^{j+1}}]) \log (C_2 \varepsilon^{-4^j/3^{j+1}}) + 12m \varepsilon^{4^j/3^{j+1}} \log (C_3Rm \varepsilon^{-5\cdot 4^j/3^{j+1}}) \right).
\end{multline*}
Recall that $h_{R,\eta}(x,\norm{\cdot}_2)$ decreases to $h(x)$ as $\eta \to 0$.  Thus,
\[
h(M) = h_R(x) \leq h_{R,\varepsilon}(x) + \sum_{j=0}^\infty \left( \mu_{|D_f(x)|}([0,C_1 \varepsilon^{4^j/3^{j+1}}]) \log (C_2 \varepsilon^{-4^j/3^{j+1}}) + 12m \varepsilon^{4^j/3^{j+1}} \log (C_3Rm \varepsilon^{-5\cdot 4^j/3^{j+1}}) \right).
\]
Of course, because $\Gamma_R^{(n)}(\mathcal{O})$ is always contained in $B_{\mathbb{M}_n(\bC)^d,\norm{\cdot}_2}(0,R)$, the first term $h_{R,\varepsilon}(x)$ is automatically finite.  The summability in $j$ of the term $12m \varepsilon^{4^j/3^{j+1}} \log (C_3Rm \varepsilon^{-5\cdot 4^j/3^{j+1}})$ in the series is straightforward: $t \log (1/t^5)$ is bounded by a constant times $t^{3/2}$, hence we can estimate the terms by a constant times $\varepsilon^{4^j/3^j}$ which is in turn bounded by a geometric series.  Thus, to complete the argument, it suffices to show the summability of the first term.  We rewrite
\[
\sum_{j=0}^\infty \mu_{|D_f(x)|}([0,C_1 \varepsilon^{4^j/3^{j+1}}]) \log (C_2 \varepsilon^{-4^j/3^{j+1}}) = \int_0^\infty \phi(t)\,d\mu_{|D_f(x)|}(t),
\]
where
\[
\phi(t) = \sum_{j=0}^\infty \log (C_2 \varepsilon^{-4^j/3^{j+1}}) \, 1_{[0,C_1 \varepsilon^{4^j/3^{j+1}}]}(t).
\]
We claim that $\phi(t) \leq A + B \log (1/t)$ for some constants $A$ and $B$ (depending on $\varepsilon$ and all the parameters in the theorem), and this claim is sufficient to complete the proof because $\mu_{|D_f(x)|}$ is a compactly supported probability measure and we assumed that $\int_0^\infty \log (1/t) \,d\mu_{|D_f(x)|}(t) < \infty$.  For every $t \in [0,C_1\varepsilon)$, there exists a unique $k \in \bN$ such that
\[
C_1 \varepsilon^{4^k/3^{k+1}} < t \leq C_1 \varepsilon^{4^{k-1}/3^k}.
\]
Then
\begin{align*}
\phi(t) &= \sum_{j=0}^{k-1} \log (C_2 \varepsilon^{-4^j/3^{j+1}}) \\
&= \sum_{j=0}^{k-1} \left( \log C_2 + \frac{4^j}{3^{j+1}} \log \frac{1}{\varepsilon} \right) \\
&\leq \left( \log C_2 + \frac{1}{3} \log \frac{1}{\varepsilon} \right) \sum_{j=0}^{k-1} \left( \frac{4}{3} \right)^j \\
&\leq 12 \left( \log C_2 + \frac{1}{3} \log \frac{1}{\varepsilon} \right) \frac{4^{k-1}}{3^k} \\
&\leq 12 \left( \log C_2 + \frac{1}{3} \log \frac{1}{\varepsilon} \right) \frac{\log (1/t) + \log C_1}{\log (1/\varepsilon)} \\
&= A + B \log \frac{1}{t},
\end{align*}
for some constants $A$ and $B$.
\end{proof}

\textbf{Remark.}
Given the apparent freedom to choose parameters in \eqref{eq: shrinking estimate}, one might wonder whether it is possible to improve the argument to allow a weaker hypothesis on $\mu_{|D_f(x)|}$ than integrability of the logarithm.  But in fact, this hypothesis is necessary for any argument based on \eqref{eq: shrinking estimate} to bound $h(x)$.  Indeed, suppose we choose a sequence $\varepsilon_k$ decreasing to zero and $t_k \in (\varepsilon_k/R, 1/3)$ and suppose that
\[
\sum_{k=0}^\infty \mu_{|D_f(x)|}\left( \left[0,\frac{C_1\varepsilon_n^2}{t_k\varepsilon_{k+1}} \right] \right) \log \frac{C_2 \varepsilon_k}{\varepsilon_{k+1}} < \infty.
\]
Since $\varepsilon_k$ is decreasing and $t_k \leq 1/3$, we have $C_1 \varepsilon_k^2 / t_k \varepsilon_{n+1} \geq 3C_2 \varepsilon_k$.  Since $\varepsilon_k < 1$, we have $\log (C_2 \varepsilon_k / \varepsilon_{k+1}) \geq \log (C_2 / \varepsilon_{k+1})$.  Hence,
\begin{align*}
\mu_{|D_f(x)|}\left( \left[0,\frac{C_1\varepsilon_k^2}{t_k\varepsilon_{k+1}} \right] \right) \log \frac{C_2 \varepsilon_k}{\varepsilon_{k+1}}
&\geq \mu_{|D_f(x)|}\left( \left(3C_1 \varepsilon_{k+1},3C_1 \varepsilon_k \right] \right) \log (C_2 / \varepsilon_{k+1}) \\
&\geq \int_{(3C_1\varepsilon_{k+1},3C_1 \varepsilon_k]} \log (3C_1 C_2/t)\,d\mu_{|D_f(x)|}(t).
\end{align*}
Hence, if the sum converges, then $\int_0^1 \log(1/t)\,\mu_{|D_f(x)|}(t) < \infty$.
% \end{rem}

\textbf{Remark.} \label{rem:tracesmooth}
Although we have stated Theorem \ref{thm: von Neumann algebraic} only for polynomial $f$ for simplicity, the same argument works for more general non-commutative functions.  Indeed, it only requires that $f$ has a Taylor expansion and error estimate as in Lemma \ref{lemm: Taylor approximation} and that the spectral measure of $|\partial f|$ is the large-$n$ limit of the spectral measures of corresponding operators on $M_n(\bC)^d$ as in Lemma \ref{lemm:weak*-weak* continuity}.  This holds for instance if $f$ is given by a non-commutative power series with radius of convergence $R' > R$ as in \cite[Section 3.3]{VoiculescuFreeEntropy2}.  More generally, it applies to the non-commutative $C^2$ functions of \cite{JLS2021} (as well as those of \cite{DGS2016}).  Roughly speaking, \cite[\S 3.2]{JLS2021} defines a space $C_{\tr}^k(\bR^{*d})$ consisting of functions $f$ that can be evaluated on self-adjoint $d$-tuples $(x_1,\dots,x_d)$ from every tracial von Neumann algebra $(M,\tau)$, such that $f$ is a Fr\'echet $C^k$ map $M_{\sa}^d \to M$, and the Fr\'echet derivatives of order $j \leq k$, viewed as multilinear maps $(M_{\sa}^d)^j \to M$, satisfy
\[
\norm{\partial^j f(x)[y_1,\dots,y_j]}_p \leq \text{constant}(f,j,R) \norm{y_1}_{p_1} \dots \norm{y_j}_{p_j}
\]
whenever $1/p = 1/p_1 + \dots + 1/p_j$ and $\norm{x}_\infty \leq R$, and such that trace polynomials are dense.  In particular, the space is cooked up so that Taylor expansions with error estimates inspired by the non-commutative H\"older's inequality, such as Lemma \ref{lemm: Taylor approximation}, will hold.  Furthermore, \cite[\S 4.4]{JLS2021} describes a trace (as well as a log-determinant for invertible elements) on the algebra $C_{\tr}^{k-1}(\bR^{*d},\mathscr{M}^1)$ in which the first derivatives $\partial_j f$ of a trace $C^k$ function $f$ live.  Extending this trace to $d \times d$ matrices over $C_{\tr}^{k-1}(\bR^{*d},\mathscr{M}^1)$ enables us to make sense of the spectral measure of $\partial f(x)^* \partial f(x)$.  This also applies to the operator $D_f f(x)$ in Theorem \ref{thm: von Neumann algebraic} since the $t_j \otimes 1 - 1 \otimes t_j$ defines an element of $C_{\tr}^{k-1}(\bR^{*d},\mathscr{M}^1(\bR^{*1}))$ for each $j$.  Furthermore, thanks to the way that the trace on $C_{\tr}^{k-1}(\bR^{*d},\mathscr{M}^1)$ describes the asymptotic behavior of traces on matrices (see \cite[\S 4.5]{JLS2021}), Lemma \ref{lemm:weak*-weak* continuity} generalizes to this setting.  Hence, \emph{mutatis mutandis} Theorem \ref{thm: von Neumann algebraic} generalizes to $f \in C_{\tr}^2(\bR^{*d})^m$.
% \end{rem}

We have now completed the proof of Theorem \ref{thm: von Neumann algebraic}. We refer the reader to \S \ref{sec: proof of Betti} for a proof that Theorem \ref{thm: von Neumann algebraic} implies Theorem \ref{thm: vanishing L2 Betti}.

\section{Connections to $L^{2}$-invariants of sofic groups} \label{sec: derivations}

In this section, we recall the connection between $\ell^2$ cohomology and the non-commutative difference quotient (\S \ref{sec: cocycles}) exploited by Shlyakhtenko \cite{Shl2015} as well as his argument why Theorem \ref{thm: von Neumann algebraic} implies Theorem \ref{thm: vanishing L2 Betti} (\S \ref{sec: proof of Betti}).   Then we show how the argument for Theorem \ref{thm: vanishing L2 Betti}, together with Shalom's result \cite{ShalomCohomolgyCharacterize}, furnishes an alternative proof of strong $1$-boundedness for the von Neumann algebras associated to sofic Property (T) groups (\S \ref{sec:PropTsofic}).

\subsection{Cocycles, derivations, and the free difference quotient} \label{sec: cocycles}

This subsection describes how to translate from group cohomology to derivations on the group algebra to the kernel of the free difference quotient $\partial f$ for a function $f$ associated to a group presentation, following \cite{ConnesShl,MinShl,ThomL2cohom,Shl2015}.

For a $*$-algebra $A$ and an $A$-$A$ bimodule $\mathcal{H},$ let $\Der(A,\mathcal{H})$ denote the set of derivations $\delta\colon A\to \mathcal{H}$.  If $(M,\tau)$ is a tracial von Neumann algebra, and $A\subseteq M$ is a weak$^{*}$-dense $*$-subalgebra, then one bimodule of interest is $L^{2}(M)\otimes L^{2}(M)$, where $A$ acts on the left by left multiplying by $a\otimes 1$ and on the right by right multiplying by $1\otimes a$.  We have a commuting action of $M \overline{\otimes} M^{\op}$ on $L^{2}(M)\otimes L^{2}(M)$ where $a\otimes b^{\op}$ acts on $c\otimes d$ by sending it to $cb\otimes ad$. We use $\#_{\operatorname{in}}$ for this action, so
\[
(a\otimes b^{\op})\#_{\operatorname{in}}(\xi)=(1\otimes a)\xi (b\otimes 1),
\]
it is straightforward to verify that this action extends to a normal representation of $M\overline{\otimes}M^{\op}$ on $L^{2}(M)\otimes L^{2}(M).$ Moreover, for all $x\in M\overline{\otimes}M^{\op},$ all $a,b\in M,$ and all $\xi\in L^{2}(M)\otimes L^{2}(M),$
\[
x\#_{\operatorname{in}}((a\otimes 1)\xi(1\otimes b))=(a\otimes 1)(x\#_{\operatorname{in}}\xi)(1\otimes b).
\]
This produces an action of $M \overline{\otimes} M^{\op}$ on $\Der(A,L^{2}(M)\otimes L^{2}(M))$ by
\[
(x\delta)(a) = x \#_{\operatorname{in}}(\delta(a))
\]
for all $x\in M\overline{\otimes}M^{\op},a\in A.$
So we may regard $\Der(A,L^{2}(M)\otimes L^{2}(M))$ as a module over $M \overline{\otimes} M^{\op}$, and so it makes sense by \cite{Luck1} to consider
\[
\dim_{M \overline{\otimes} M^{\op}}(\Der(A,L^{2}(M)\otimes L^{2}(M))).
\]
We have a special class of derivatives called the \emph{inner derivations}. We say that $\delta$ is inner if there is a $\xi\in L^{2}(M)\otimes L^{2}(M)$ with $\delta(a)=[a,\xi].$ We let $\textnormal{Inn}(A,L^{2}(M)\otimes L^{2}(M))$ be the inner derivations, and let
\[H^{1}(A,\tau)=\frac{\Der(A,L^{2}(M)\otimes L^{2}(M))}{\textnormal{Inn}(A,L^{2}(M)\otimes L^{2}(M)}.\]
We define the first $L^{2}$-Betti number of $A$ by
\[\beta^{1}_{(2)}(A,\tau)=\dim_{M\overline{\otimes}M^{\op}}(H^{1}(A,\tau))\]
This definition is due to Connes-Shlyakhtenko \cite{ConnesShl}.

\begin{prop}\label{prop: switching from groups to algebras}
Let $G$ be a countable, discrete group, let $\tau$ be the canonical trace, and set $M = L(G).$ Then
\begin{enumerate}[(i)]
    \item $\beta^{1}_{(2)}(G)=\beta^{1}_{(2)}(\bC[G],\tau).$ In particular, if $G$ is infinite, then
    \[ \beta^{1}_{(2)}(G)+1=\dim_{M\overline{\otimes}M^{\op}}(\Der(\bC[G],L^{2}(M)\otimes L^{2}(M)))\]
    \label{item: L2 betti in terms of deriv}
    \item Suppose that $G$ is finitely generated, and suppose $g_{1},\cdots,g_{r}$ is a finite generating set. Set
    \[x=(\textnormal{Re}(g_{1}),\textnormal{Im}(g_{1}),\textnormal{Re}(g_{2}),\textnormal{Im}(g_{2}),\cdots,\textnormal{Re}(g_{k}),\textnormal{Im}(g_{k}))\in (\bC[G]_{\sa})^{2r}.\]
    where $\textnormal{Re}(a)=\frac{a+a^{*}}{2},\textnormal{Im}(a)=\frac{a-a^{*}}{2i}$ for all $a\in \bC[G].$ Let $J$ be the kernel of the homomorphism
    \[\ev_{x}\colon \bC\ang{t_{1},\cdots,t_{2r}}\to \bC[G].\]
    Then $G$ is finitely presented if and only if $J$ is finitely generated as a two-sided ideal. \label{I:finite presentation in terms of *-alg}
\end{enumerate}
\end{prop}

\begin{proof}
(\ref{item: L2 betti in terms of deriv}): This is \cite[Proposition 2.3]{ConnesShl},\cite[Corollary 3.6]{MinShl}, \cite[Section 4]{ThomL2cohom}.

 (\ref{I:finite presentation in terms of *-alg}):
Let $\bF_{r}$ be the free group on letters $a_{1},\cdots,a_{r}.$ Consider the surjective homomorphism $q\colon \bF_{r}\to G$ so that $q(a_{j})=g_{j},$ we continue to use $q$ to denote the linear extension $q\colon \bC[\bF_{r}]\to \bC[G].$ Let
\[y=(\textnormal{Re}(a_{1}),\textnormal{Im}(a_{1}),\textnormal{Re}(a_{2}),\textnormal{Im}(a_{2}),\cdots,\textnormal{Re}(a_{k}),\textnormal{Im}(a_{k}))\in (\bC[\bF_{r}])_{\sa}^{2r},\]
so $\ev_{x}=q\circ \ev_{y}.$ Let $B$ be the ideal in $\bC\ang{t_{1},\cdots,t_{2k}}$ generated by $\{[t_{2j-1},t_{2j}]:j=1,\cdots,k\}\cup\{t_{2j-1}^{2}+t_{2j}^{2}-1:j=1,\cdots,k\},$ and let
\[\pi\colon \bC\ang{t_{1},\cdots,t_{k}}\to \bC\ang{t_{1},\cdots,t_{k}}/B\]
be the quotient map. Then the kernel of $\ev_{y}$ contains $\pi$ and so $\ev_{y}$ descends to a map
\[\overline{\ev}_{y}\colon \bC\ang{t_{1},\cdots,t_{2k}}/B\to \bC[\bF_{r}]\]
with $\ev_{y}=\overline{\ev}_{y}\circ \pi.$ For every $1\leq j\leq k,$ the element $u_{j}=t_{2j-1}+it_{2j}+B\in \bC\ang{t_{1},\cdots,t_{2k}}/B$ is unitary, and so there is a unique map $\phi\colon \bC[\bF_{r}]\to \bC\ang{t_{1},\cdots,t_{2k}}/B$ which satisfies $\phi(a_{j})=u_{j}.$ Routine calculations verify that $\phi,\overline{\ev}_{y}$ are mutual inverses to each other, and so $\phi \circ \ev_{y}=\pi.$

First suppose that $G$ is finitely presented, and let $F$ be a finite subset of the kernel of $q: \bF_r \to G$ so that $\ker(q)$ is the smallest normal subgroup containing $F.$ It is direct to verify that the kernel of $q: \bC[\bF_r] \to \bC[G]$ is the smallest ideal in $\bC[\bF_{r}]$ containing $\{w-1:w\in F\}.$ For $w\in F,$ let $Q_{w}\in \bC\ang{t_{1},\cdots,t_{2k}}$ be any element so that $\pi(Q_{w})=\phi(w).$ We leave it as an exercise to show that $J$ is generated as a two-sided ideal by
\[\{Q_{w}-1:w\in F\}\cup \{[t_{2j-1},t_{2j}]:j=1,\cdots,k\}\cup\{t_{2j-1}^{2}+t_{2j}^{2}-1:j=1\}.\]
% We claim that $J$ is generated as a two-sided ideal by
% \[\{Q_{r}-1:r\in R\}\cup \{[t_{2j-1},t_{2j}]:j=1,\cdots,k\}\cup\{t_{2j-1}^{2}+t_{2j}^{2}-1:j=1,\cdots,k\}.\]
% Indeed if $P\in \bC\ang{t_{1},\cdots,t_{2k}}$ and $\ev_{x}(P)=0,$ then $q(\ev_{y}(P))=0,$ and so
% \[\ev_{y}(P)=\sum_{r\in R}a_{r}(r-1)b_{r}\]
% for some $(a_{r})_{r},(b_{r})_{r}\in \bC[\bF_{r}]^{R}.$ So
% \[\pi(P)=\sum_{r\in R}\phi(a_{r})(\pi(Q_{r})-1)\pi(b_{r}).\]
% Let $(A_{r})_{r},(B_{r})_{r}\in \bC\ang{t_{1},\cdots,t_{2k}}^{R}$ be so that $\pi(A_{r})=a_{r},\pi(B_{r})=b_{r}.$ Then
% \[P-\sum_{r\in R}A_{r}(Q_{r}-1)B_{r}\in B\]
% and so $P$ is in the two-sided ideal generated by
% \[\{Q_{r}-1:r\in R\}\cup \{[t_{2j-1},t_{2j}]:j=1,\cdots,k\}\cup\{t_{2j-1}^{2}+t_{2j}^{2}-1:j=1,\cdots,k\}.\]
This shows that $J$ is finitely generated as a two-sided ideal.

Now suppose that $J$ is finitely generated as a two-sided ideal, say by $F_{1},\cdots,F_{k}.$ Set $N=\ker(q\colon G\to \bF_{r}),$ and $Q_{j}=\ev_{y}(F_{j}).$ Then $\pi(F_{1}),\cdots,\pi(F_{k})$ generate $\ker(q\circ \overline{\ev}_{y})$ as a two-sided ideal. Since $\overline{\ev}_{y}$ is an isomorphism, it follows that $Q_{1},\cdots,Q_{k}$ generate $\ker(q\colon \bC[\bF_{r}]\to \bC[G])$ as a two-sided ideal. Observe that $\ker(q\colon \bC[\bF_{r}]\to \bC[G])$ is generated as a two-sided ideal by $\{x-1:x\in N\}.$ So for $j=1,\cdots,k$ we can find a finite $F_{j}\subseteq N$ so that $Q_{j}$ is in the two-sided ideal generated by $\{x-1:x\in F_{j}\}.$ Let $F=\bigcup_{j=1}^{k}F_{j},$ and let $I$ be the two-sided in $\bC[\bF_{r}]$ generated by $\{x-1:x\in F\}.$
Then $Q_{j}\in I$ for all $j,$ and so $I=\ker(q\colon \bC[\bF_{r}]\to \bC[G]).$ If $\widetilde{N}$ is the normal subgroup of $G$ generated by $R,$ then $I$ is the kernel of the natural quotient map
\[\bC[\bF_{r}]\to \bC[\bF_{r}/\widetilde{N}].\]
But $\widetilde{N}\leq N,$ and we saw above that $I$ is the kernel of the natural quotient map
\[
\bC[\bF_{r}]\to \bC[\bF_{r}/N].
\]
So $N = \widetilde{N}$, and this establishes that $G$ is finitely presented.
\end{proof}

The following may be argued exactly as in \cite[Lemma 3.1]{Shl2015}.

\begin{prop}\label{prop: derivations are a variety}
Let $(M,\tau)$ be a tracial von Neumann algebra and let $x\in M_{\sa}^{k}$ be such that $W^{*}(x)=M$.  Let $A$ be the $*$-algebra generated by $x,$ and let $J$ be the kernel of $\ev_{x}\colon \bC\ang{t_{1},\cdots,t_{k}}\to A.$ Suppose that $(F_{j})_{j=1}^{\infty}$ is a sequence which generates $J$ as a two-sided ideal in $\bC\ang{t_{1},\cdots,t_{k}}$. Then the map
\[\delta\mapsto (\delta(x_{j}))_{j=1}^{k}\]
is an $M$-$M$ bimodular isomorphism
\[
\Der(A,L^{2}(M)\otimes L^{2}(M))\to \bigcap_{j=1}^{\infty}\ker((\partial F_{j})(x)\#).
\]
\end{prop}

\subsection{Strong $1$-boundedness from vanishing $\ell^2$-Betti numbers} \label{sec: proof of Betti}
For this subsection, we need the following notation. Given a group $G$, we view $\bC(G)\subseteq L(G)$ by sending $\sum_{g}a_{g}g\to \sum_{g}a_{g}\lambda(g)$. This induces natural inclusion $\bM_{m,n}(\bC(G))\subseteq \bM_{m,n}(L(G))$. Given $A\in \bM_{m,n}(\bC(G))$, we let $\mu_{|A|}$ be the spectral measure of $(A^{*}A)^{1/2}$ with respect to the trace $\Tr\otimes \tau$, with $\tau$ defined as in Section \ref{sec:preliminary}. We define
\[\Det^{+}_{L(G)}(A)=\exp\left(\int_{(0,\infty)}\log(t)\,d\mu_{|A|}(t)\right).\]

We have explained how to get from $\ell^2$ Betti number conditions as in Theorem \ref{thm: vanishing L2 Betti} to conditions on $\partial f$ for some tuple $f$ of non-commutative polynomials as in Theorem \ref{thm: von Neumann algebraic}.  The other main ingredient needed to prove Theorem \ref{thm: vanishing L2 Betti} is positivity of Fuglede-Kadison determinant.  The following theorem of Elek and Szabo is the main way we know of to guarantee positivity of Fuglede-Kadison determinants.

\begin{thm}[Theorem 5 in \cite{ElekSzaboDeterminant}]\label{thm:ES Det}
Let $G$ be a countable, discrete, sofic group, and $m,n\in \bN$. Fix $A\in \bM_{m,n}(\bZ(G))$. Then
\[
\Det_{L(G)}^{+}(A)\geq 1.
\]
\end{thm}

Note that if $G$ is as in the statement of the above theorem, and $A\in \bM_{m,n}(\bQ(G))$ for some $m,n\in \bN$, then there is a $q\in \bN$ so that $qA\in M_{m,n}(\bZ(G))$. Thus
\[\Det_{L(G)}^{+}(A)=\frac{1}{q}\Det_{L(G)}^{+}(qA)\geq \frac{1}{q}>0.\]
Having collected the appropriate background material on derivations and $L^{2}$-Betti numbers, we now discuss why Theorem \ref{thm: von Neumann algebraic} implies Theorem \ref{thm: vanishing L2 Betti}.

\begin{proof}[Proof of Theorem \ref{thm: vanishing L2 Betti} from Theorem \ref{thm: von Neumann algebraic}]
Let $G=\langle{g_{1},\cdots,g_{s}|w_{1},\cdots,w_{l}\rangle}$ be a finite presentation of $G$. For $1\leq j\leq s$, set
\[x_{2j-1}=\frac{g_{j}+g_{j}^{-1}}{2},\,\,\ x_{2j}=\frac{g_{j}-g_{j}^{-1}}{2i},\]
and set $x=(x_{1},x_{2},\cdots,x_{2s})\in (\bC[G]_{\sa})^{2k}$. Let $q\colon \bC[\bF_{r}]\to \bC[G]$ and $y\in \bC[\bF_{r}]_{\sa}^{2k}$ be as in the proof of Proposition \ref{prop: switching from groups to algebras} (\ref{I:finite presentation in terms of *-alg}).  For $j=1,\cdots,l+2s$ define $f_{j}\in \bC\langle{t_{1},\cdots,t_{2k}\rangle}$ by
\[f_{j}=\begin{cases}
w_{j}(t_{1}+it_{2},t_{3}+it_{4},\cdots,t_{2s-1}+it_{2s}),& \textnormal{ if $1\leq j\leq l$}\\
t_{2j-1}t_{2j}-t_{2j}t_{2j-1},& \textnormal{ if $l+1\leq j\leq l+s$}\\
t_{2j-1}^{2}+t_{2j}^{2}-1,& \textnormal{ if $l+s+1\leq j\leq l+2s$.}
\end{cases}\]
By the proof of Proposition \ref{prop: switching from groups to algebras} (\ref{I:finite presentation in terms of *-alg}), we see that the kernel of $\ev_{x}\colon \bC\langle{t_{1},\cdots,t_{2s}}\to \bC[G]$ is generated (as an ideal) by
\[\{f_{1},f_{2},\cdots,f_{l+2s}\}.\]
Set $f=(f_{1},\cdots,f_{l+2s})$. Let $D_{f}$ be as in the statement of Theorem \ref{thm: von Neumann algebraic}. We leave it as an exercise to verify that $D_{f}\in M_{l+2s+1,2s}(\bQ(G\times G))$. By Theorem \ref{thm:ES Det}, we have that $\Det_{L(G)}^{+}(D_{f})>0$, i.e. $\int_{(0,\infty)}\log (t)\,d\mu_{|D_{f}|}(t)>-\infty$. All that remains is to verify that $D_{f}$ is injective. Recall that the $(1,j)$ entry of $D_f$ is $x_j \otimes 1 - 1 \otimes x_j$ and the remaining rows are given by the matrix of partial derivatives $\partial f$ discussed in \S \ref{subsec: DQ}.
% note that $\partial f(x) \in \mathbb{M}_{l+2s,2s}(\bQ[G])$.
Suppose that $\xi \in [L^{2}(M)\otimes L^{2}(M)]^{2s}$ and $D_{f}\#\xi =0$. This implies that $(\partial f)(x)\#\xi =0$. By Proposition  \ref{prop: derivations are a variety} we see that there is a derivation $\delta\colon \bC[G]\to L^{2}(M)\otimes L^{2}(M)$ so that $\xi_{j}=\delta(x_{j})$ for $j=1,\cdots,2s$. By Proposition \ref{prop: switching from groups to algebras} and the fact that $\beta^{1}_{(2)}(G)=0$, we find that $\delta$ is approximately inner. Thus we may choose a sequence $\zeta_{n}\in L^{2}(M)\otimes L^{2}(M)$ so that for all $j=1,\cdots, 2s$
\[\xi_{j}=\lim_{n\to\infty}[x_{j},\zeta_{n}]=\lim_{n\to\infty}(x_{j}\otimes 1-1\otimes x_{j}^{op})\#\zeta_{n}. \]
Since $D_{f}\#\xi=0$, we have that $(x_{j}\otimes 1-1\otimes x_{j}^{op})\#\xi_{j}=0$ for all $j=1,\cdots,2s$. Thus, for all $j=1,\cdots,2s:$
\[\|\xi_{j}\|_{2}^{2}=\lim_{n\to\infty}\langle{\xi_{j},(x_{j}\otimes 1-1\otimes x_{j}^{op})\#\zeta_{n}\rangle}=\lim_{n\to\infty}\langle{(x_{j}\otimes 1-1\otimes x_{j}^{op})\#\xi_{j},\zeta_{n}\rangle}=0.\]
So we have shown that $\xi=0$. Thus $D_{f}$ is injective, and this completes the proof.
\end{proof}

More generally, the same proof shows that if $(A,\tau)$ is any tracial $*$-algebra and
\begin{itemize}
    \item $\beta^{1}_{(2)}(A,\tau)=0$,
    \item there exists a generating tuple $x\in A_{\sa}^{d}$ and $f\in \bC\langle{t_{1},\cdots,t_{d}\rangle}^{\oplus m}$ so that $\{f_{1},\cdots,f_{m}\}$ generates $\ev_{x}$ as an ideal, and with $\Det_{A}((\partial f)(x))>0$,
\end{itemize}
then $W^{*}(A,\tau)$ is strongly $1$-bounded. This recovers the case $n=\rank((\partial F)(x))$ of \cite[Theorem 2.5]{Shl2015}.

\subsection{Strong $1$-boundedness of Property (T) sofic groups from Theorem \ref{thm: vanishing L2 Betti}} \label{sec:PropTsofic}

The vanishing of first $\ell^2$-Betti numbers for Property (T) groups was obtained in \cite{Bekka1997GroupCH} (see Corollary 6). %In this section we use results of Shalom \cite{ShalomCohomolgyCharacterize} and Shlyakhtenko \cite{Shl2015}
We will need a little more than the above result to give a short proof that sofic groups with Property (T) are strongly $1$-bounded.
Specifically, will need the full strength of the Delorme-Guichardet Theorem \cite{DelormeT, GuichardetT}, which is about cohomology of groups with values in a unitary representation.
This is because we will need not just the cohomology with values in the left regular representation of a group, but in the quasi-regular representation on $\ell^{2}$ of a coset spaces.
Let $G$ be a countable, discrete group and $\pi\colon G\to \mathcal{U}(\mathcal{H})$ a unitary representation. A \emph{cocycle for $\pi$} is a map $\beta\colon G\to \mathcal{H}$ which satisfies
\[\beta(gh)=\pi(g)\beta(h)+\beta(g)\mbox{ for all $g,h\in G$.}\]
We say that $\beta$ is \emph{inner} if there is a $\xi\in \mathcal{H}$ so that $\beta(g)=\pi(g)\xi-\xi$. The Delorme-Guichardet theorem says that $G$ has (T) if and only if for every cocycle on $G$ with values in a unitary representation is inner. See \cite[Section 2.12]{BHV} for a proof.

\begin{lemm}
Let $\widetilde{G},G$ be  Property (T) groups and let $q\colon \widetilde{G}\to G$ be a surjective homomorphism. Let $\mathcal{H}$ be an $L(G)-L(G)$ bimodule, and view $\mathcal{H}$ as a bimodule over $\bC[\widetilde{G}]$ via $q$. Then every derivation $\delta\colon \bC[\widetilde{G}]\to \mathcal{H}$ is inner.
\end{lemm}

\begin{proof}
% Consider the representation $\pi\colon \widetilde{G}\to \mathcal{U}(\mathcal{H})$ given by $\pi(x)\xi=u_{q(x)}\xi u_{q(x)}^{-1}$, where $(u_{g})_{g\in G}$ are the canonical unitaries in $G$ generating $L(G)$. First note that $\pi$ does not have nonzero fixed vectors. Indeed, if $\xi\in \mathcal{H}$ is fixed under $\pi$, then $u_{x}\xi=\xi u_{x}$ for all $x\in G$ and normality then implies that $\xi$ is an $L(G)$-central vector. Our hypotheses thus imply that $\xi=0$.

% Since $\pi$ does not have fixed vectors and $\widetilde{G}$ has (T), we know that $\pi$ does not have nontrivial almost invariant vectors. Thus $H^{1}(\widetilde{G},\pi)=\overline{H}^{1}(\widetilde{G},\pi)$, and since $\widetilde{G}$ has (T), a theorem of Shalom implies that $H^{1}(\widetilde{G},\pi)=0$.

Suppose that $\delta\colon \bC[\widetilde{G}]\to \mathcal{H}$ is a derivation. Define $\beta\colon \widetilde{G}\to \mathcal{H}$ by $\beta(x)=\delta(x)u_{q(x)}^{-1}$.  The fact that $\delta$ is a derivation implies, by a direct calculation, that $\beta$ is a cocycle for $\pi$. By the Delorme-Guichardet theorem and the fact that $\widetilde{G}$ has Property (T) we know that $\beta$ is inner, i.e. there is a $\xi\in\mathcal{H}$ so that $\beta(x)=u_{q(x)}\xi u_{q(x)}^{-1}-\xi$ for all $x\in \widetilde{G}$. So for all $x\in \widetilde{G}$
\[\delta(x)=\beta(x)u_{q(x)}=u_{q(x)}\xi-\xi u_{q(x)},\]
and this  verifies that $\delta$ is inner.

\end{proof}

We will primarily interested in the following special case of the above lemma.

\begin{cor}
Let $\widetilde{G},G$ be infinite Property (T) groups and let $q\colon \widetilde{G}\to G$ be a surjective homomorphism. Set $M=L(G)$. Then every derivation $\delta\colon \bC[\widetilde{G}]\to L^{2}(M)\otimes L^{2}(M)$ is inner.
\end{cor}

% \begin{proof}
% Since $G$ is infinite, we know that $M$ is diffuse, and so $L^{2}(M)\otimes L^{2}(M)$ has no $M$-central vectors. So the preceding lemma implies the result.

% \end{proof}
We now show that Property (T) sofic groups are strongly $1$-bounded. We argue directly from \cite{Shl2015} using a Theorem of Shalom on the structure of Property (T) groups.

\begin{cor}
Let $G$ be an infinite Property (T) sofic group. Then $L(G)$ is strongly $1$-bounded.

\end{cor}

\begin{proof}
Since $G$ has Property (T), it is finitely generated. By a theorem of Shalom \cite[Theorem 6.7]{ShalomCohomolgyCharacterize}, there is a finitely presented Property (T) group $\widetilde{G}$ and a surjective homomorphism $q\colon \widetilde{G}\to G$. It may be that $\widetilde{G}$ is not sofic. However, we will still be able to use soficity of $G$ to apply Shlyakhtenko's results to our setting.

Let $\widetilde{S}$ be a finite generating set of $\widetilde{G}$ and set $S=q(\widetilde{S})$. Then there is a finite set $R$ of words in $S$ so that $\widetilde{G}$ has a presentation $\langle{S|R\rangle}$. Use $S$ to build self-adjoint generators $x=(x_{1},\cdots,x_{m})$ of $\bC[G]$ which have lifts $\widetilde{x}=(\widetilde{x}_{1},\cdots,\widetilde{x}_{r})$ to generators of $\widetilde{G}$. Now use the relations $R$ to produce $F_{1},\cdots,F_{r}\in \bQ[i]\langle{t_{1},\cdots,t_{m}\rangle}$ with the property that if $J$ is the ideal generated by $F_{1},\cdots, F_{r}$, then the natural map $\bC\langle{t_{1},\cdots,t_{r}\rangle}\to \bC[\widetilde{G}]$ given by $F\mapsto F(\widetilde{x})$ has kernel $J$. Let $F=(F_{1},\cdots,F_{r})$. By the proof of Proposition \ref{prop: derivations are a variety}, we have that
\[\ker((\partial F)(x)\#)\cong \Der(\bC[\widetilde{G}],L^{2}(M)\otimes L^{2}(M))\]
with $M = L(G)$. By the preceding corollary, it follows that $\ker((\partial F)(x))$ corresponds under this isomorphism to the inner derivations $\bC[\widetilde{G}]\to L^{2}(M)\otimes L^{2}(M)$, and since $M$ is diffuse
\[\dim_{M\overline{\otimes}M^{op}}(\ker((\partial F)(x)\#))=\dim_{M\overline{\otimes}M^{op}}(\Inn(\bC[\widetilde{G}],L^{2}(M)\otimes L^{2}(M)))=1.\]
Further, since $F_{1},\cdots,F_{r}\in \bQ[i]\langle{t_{1},\cdots,t_{r}\rangle}$, we know from soficity of $G$ and Theorem \ref{thm:ES Det} that $\det_{M}^{+}((\partial F)(x))>0$. Thus a theorem of Shlyakhtenko \cite{Shl2015}  implies that $M$ is strongly $1$-bounded (this also follows from our proof of Theorem \ref{thm: vanishing L2 Betti} from Theorem \ref{thm: von Neumann algebraic}, see the discussion at the end of the previous subsection).
\end{proof}

\begin{ack}
This project began when the first and third authors met in the conference titled ``Classification of von Neumann algebras" held at BIRS. We thank the organizers for arranging this conference, and BIRS for providing a stimulating research environment.

\end{ack}

\begin{funding}
B. Hayes gratefully acknowledges support from the NSF grant DMS-2000105.  D. Jekel was supported by NSF grant DMS-2002826.
\end{funding}
%
%\bibliographystyle{emss}
%
%\bibliography{InnerAmen}

\begin{thebibliography}{10}
\providecommand{\url}[1]{\texttt{#1}}
\providecommand{\urlprefix}{URL }
\providecommand{\eprint}[2][]{\url{#2}}

\bibitem{AGZ2009}
G.~W. Anderson, A.~Guionnet, and O.~Zeitouni, \emph{An introduction to random
  matrices}. Cambridge Studies in Advanced Mathematics, Cambridge University
  Press, 2009

\bibitem{BHV}
B.~Bekka, P.~de~la Harpe, and A.~Valette, \emph{Kazhdan's property ({T})}. New
  Mathematical Monographs 11, Cambridge University Press, Cambridge, 2008
  \MR{2415834}

\bibitem{Bekka1997GroupCH}
M.~E.~B. Bekka and A.~Valette, Group cohomology, harmonic functions and the
  first l2 -betti number. \emph{Potential Analysis} \textbf{6} (1997), 313--326

\bibitem{BCG2003}
P.~Biane, M.~Capitaine, and A.~Guionnet, Large deviation bounds for matrix
  {B}rownian motion. \emph{Inventiones Mathematicae} \textbf{152} (2003),
  433--459

\bibitem{BrannanVergnioux}
M.~Brannan and R.~Vergnioux, Orthogonal free quantum group factors are strongly
  1-bounded. \emph{Adv. Math.} \textbf{329} (2018), 133--156 \MR{3783410}

\bibitem{ConnesShl}
A.~Connes and D.~Shlyakhtenko, {$L^2$}-homology for von {N}eumann algebras.
  \emph{J. Reine Angew. Math.} \textbf{586} (2005), 125--168 \MR{2180603}

\bibitem{DGS2016}
Y.~Dabrowski, A.~Guionnet, and D.~Shlyakhtenko, Free transport for convex
  potentials. \emph{arXiv:1701.00132}  (2016)

\bibitem{DelormeT}
P.~Delorme, {$1$}-cohomologie des repr\'{e}sentations unitaires des groupes de
  {L}ie semi-simples et r\'{e}solubles. {P}roduits tensoriels continus de
  repr\'{e}sentations. \emph{Bull. Soc. Math. France} \textbf{105} (1977),
  no.~3, 281--336 \MR{578893}

\bibitem{Dixmier1953}
J.~Dixmier, Formes lin{\'e}aires sur un anneau d'op{\'e}rateurs. \emph{Bulletin
  de la Soci{\'e}t{\'e} Math{\'e}matique de France} \textbf{81} (1953), 9--39

\bibitem{DykemaFreeEntropy}
K.~J. Dykema, Two applications of free entropy. \emph{Math. Ann.} \textbf{308}
  (1997), no.~3, 547--558 \MR{1457745}

\bibitem{ElekSzaboDeterminant}
G.~Elek and E.~Szab\'{o}, Hyperlinearity, essentially free actions and
  {$L^2$}-invariants. {T}he sofic property. \emph{Math. Ann.} \textbf{332}
  (2005), no.~2, 421--441 \MR{2178069}

\bibitem{GJNS2021}
W.~Gangbo, D.~Jekel, K.~Nam, and D.~Shlyakhtenko, Duality for optimal couplings
  in free probability. \emph{Communications in Mathematical Physics}
  \textbf{396} (2022), no.~3, 903--981

\bibitem{GePrime}
L.~Ge, Applications of free entropy to finite von {N}eumann algebras. {II}.
  \emph{Ann. of Math. (2)} \textbf{147} (1998), no.~1, 143--157 \MR{1609522}

\bibitem{GuichardetT}
A.~Guichardet, Sur la cohomologie des groupes topologiques. {II}. \emph{Bull.
  Sci. Math. (2)} \textbf{96} (1972), 305--332 \MR{340464}

\bibitem{Hayes2018}
B.~Hayes, 1-bounded entropy and regularity problems in von {N}eumann algebras.
  \emph{Int. Math. Res. Not. IMRN}  (2018), no.~1, 57--137 \MR{3801429}

\bibitem{PropTS1B}
B.~Hayes, D.~Jekel, and S.~Kunnawalkam~Elayavalli, Property ({T}) and strong
  1-boundedness for von neumann algebras. \emph{arXiv:2107.03278}

\bibitem{JLS2021}
D.~Jekel, W.~Li, and D.~Shlyakhtenko, Tracial smooth functions of non-commuting
  variables and the free {W}asserstein manifold. \emph{Dissertationes Math.}
  \textbf{580} (2022), 150 \MR{4451910}

\bibitem{JungL2B}
K.~Jung, The rank theorem and ${L}^{2}$-invariants in free entropy: Global
  upper bounds. \emph{arXiv:1602.04726}

\bibitem{Jung2007}
K.~Jung, Strongly $1$-bounded von {N}eumann algebras. \emph{Geom. Funct. Anal.}
  \textbf{17} (2007), no.~4, 1180--1200 \MR{2373014}

\bibitem{JungS}
K.~Jung and D.~Shlyakhtenko, Any generating set of an arbitrary property {T}
  von {N}eumann algebra has free entropy dimension {$\le1$}. \emph{J.
  Noncommut. Geom.} \textbf{1} (2007), no.~2, 271--279 \MR{2308307}

\bibitem{KadisonRingroseII}
R.~V. Kadison and J.~R. Ringrose, \emph{Fundamentals of the theory of operator
  algebras. {V}ol. {II}}. Graduate Studies in Mathematics 16, American
  Mathematical Society, Providence, RI, 1997 \MR{1468230}

\bibitem{Luck1}
W.~L{\"{u}}ck, Dimension theory of arbitrary modules over finite von {N}eumann
  algebras and {$L^2$}-{B}etti numbers. {I}. {F}oundations. \emph{J. Reine
  Angew. Math.} \textbf{495} (1998), 135--162 \MR{1603853}

\bibitem{MinShl}
I.~Mineyev and D.~Shlyakhtenko, Non-microstates free entropy dimension for
  groups. \emph{Geom. Funct. Anal.} \textbf{15} (2005), no.~2, 476--490
  \MR{2153907}

\bibitem{ShalomCohomolgyCharacterize}
Y.~Shalom, Rigidity of commensurators and irreducible lattices. \emph{Invent.
  Math.} \textbf{141} (2000), no.~1, 1--54 \MR{1767270}

\bibitem{Shl2015}
D.~Shlyakhtenko, Von neumann algebras of sofic groups with $\beta^{(2)}_{1}=0$
  are strongly 1-bounded. \emph{J. Operator Theory} \textbf{85} (2021), no.~1,
  217--228

\bibitem{Szarek}
S.~a.~J. Szarek, Metric entropy of homogeneous spaces. In \emph{Quantum
  probability ({G}da\'{n}sk, 1997)}, pp. 395--410, Banach Center Publ. 43,
  Polish Acad. Sci. Inst. Math., Warsaw, 1998 \MR{1649741}

\bibitem{ThomL2cohom}
A.~Thom, {$L^2$}-cohomology for von {N}eumann algebras. \emph{Geom. Funct.
  Anal.} \textbf{18} (2008), no.~1, 251--270 \MR{2399103}

\bibitem{VoiculescuV}
D.~Voiculescu, The analogues of entropy and of {F}isher's information measure
  in free probability theory. {V}. {N}oncommutative {H}ilbert transforms.
  \emph{Invent. Math.} \textbf{132} (1998), no.~1, 189--227 \MR{1618636}

\bibitem{VoiculescuFreeEntropy2}
D.-V. Voiculescu, The analogues of entropy and of {F}isher's information
  measure in free probability theory {II}. \emph{Invent. Math.} \textbf{118}
  (1994), no.~3, 411--440 \MR{1296352}

\bibitem{Voiculescu1996}
D.-V. Voiculescu, The analogues of entropy and of {F}isher's information
  measure in free probability theory. {III}. {T}he absence of {C}artan
  subalgebras. \emph{Geom. Funct. Anal.} \textbf{6} (1996), no.~1, 172--199
  \MR{1371236}

\end{thebibliography}

\end{document}